%% file: Nadirashvili-Friedlander_finalv3.tex
\documentclass[12pt]{amsart}

\usepackage{amssymb, amsfonts}
\usepackage{url}
\usepackage[all,arc]{xy}
\usepackage{amscd}
\usepackage{mathrsfs}
\usepackage{geometry}
\usepackage[colorlinks,urlcolor=blue]{hyperref}
\usepackage{comment}
\usepackage{enumerate}
\usepackage{graphicx}
\usepackage{bm}
\usepackage{color}
\usepackage{anyfontsize}
\usepackage{caption}

\numberwithin{equation}{section}

\theoremstyle{plain}
\newtheorem{theorem}{Theorem}[section]
\newtheorem{corollary}[theorem]{Corollary}
\newtheorem{lemma}[theorem]{Lemma}
\newtheorem{proposition}[theorem]{Proposition}
\newtheorem{conjecture}[theorem]{Conjecture}
\newtheorem{definition}[theorem]{Definition}

\theoremstyle{remark}
\newtheorem{remark}{Remark}[section]

\begin{document}

\date{\today} 

\title{On the Friedlander-Nadirashvili invariants of surfaces}

\author{Mikhail Karpukhin, Vladimir Medvedev}

\address{Mathematics 253-37, Caltech, Pasadena, CA 91125, USA
}
\email{mikhailk@caltech.edu}



\address{D\'{e}partement de Math\'{e}matiques et de Statistique, Pavillon Andr\'{e}-Aisenstadt, Universit\'{e} de Montr\'{e}al, Montr\'{e}al, QC, H3C 3J7, Canada \newline {\em and}
\newline Peoples' Friendship University of Russia (RUDN University), 6 Miklukho-Maklaya Street, Moscow, 117198, Russian Federation \newline {\em and}
\newline Faculty of Mathematics, National Research University Higher School of Economics, 6 Usacheva Street, Moscow, 119048, Russian Federation}

\email{medvedevv@dms.umontreal.ca}



\begin{abstract}
Let $M$ be a closed smooth manifold. In 1999, L.~Friedlander and N.~Nadirashvili introduced a new differential invariant $I_1(M)$ using  the first normalized nonzero eigenvalue of the Lalpace-Beltrami operator $\Delta_g$ of a Riemannian metric $g$. They defined it taking the  supremum of this quantity over all Riemannian metrics in each conformal class, and then taking the infimum over all conformal classes. By analogy we use $k$-th eigenvalues of $\Delta_g$ to define the invariants $I_k(M)$ indexed by positive integers $k$. In the present paper the values of these invariants on surfaces are investigated. We show that $I_k(M)=I_k(\mathbb{S}^2)$ unless $M$ is a non-orientable surface of even genus. For orientable surfaces and $k=1$ this was earlier shown by R.~Petrides. In fact L.~Friedlander and N.~Nadirashvili suggested that $I_1(M)=I_1(\mathbb{S}^2)$ for any surface $M$ different from $\mathbb{RP}^2$. We show that, surprisingly enough, this is not true for non-orientable surfaces of even genus, for such surfaces one has $I_k(M)>I_k(\mathbb{S}^2)$. We also discuss the connection  between the Friedlander-Nadirashvili invariants and the theory of cobordisms, and conjecture that  $I_k(M)$ is a cobordism invariant.


\end{abstract}

\maketitle


\newcommand\cont{\operatorname{cont}}
\newcommand\diff{\operatorname{diff}}

\newcommand{\dvol}{\text{dA}}
\newcommand{\GL}{\operatorname{GL}}
\newcommand{\myO}{\operatorname{O}}
\newcommand{\myP}{\operatorname{P}}
\newcommand{\eye}{\operatorname{Id}}
\newcommand{\myF}{\operatorname{F}}
\newcommand{\Vol}{\operatorname{Vol}}
\newcommand{\odd}{\operatorname{odd}}
\newcommand{\even}{\operatorname{even}}
\newcommand{\ol}{\overline}
\newcommand{\mye}{\operatorname{E}}
\newcommand{\myo}{\operatorname{o}}
\newcommand{\myt}{\operatorname{t}}
\newcommand{\irr}{\operatorname{Irr}}
\newcommand{\mydiv}{\operatorname{div}}
\newcommand{\re}{\operatorname{Re}}
\newcommand{\im}{\operatorname{Im}}
\newcommand{\can}{\operatorname{can}}
\newcommand{\scal}{\operatorname{scal}}
\newcommand{\tr}{\operatorname{trace}}
\newcommand{\sgn}{\operatorname{sgn}}
\newcommand{\SL}{\operatorname{SL}}
\newcommand{\myspan}{\operatorname{span}}
\newcommand{\mydet}{\operatorname{det}}
\newcommand{\SO}{\operatorname{SO}}
\newcommand{\SU}{\operatorname{SU}}
\newcommand{\specl}{\operatorname{spec_{\mathcal{L}}}}
\newcommand{\fix}{\operatorname{Fix}}
\newcommand{\id}{\operatorname{id}}
\newcommand{\grad}{\operatorname{grad}}
\newcommand{\singsup}{\operatorname{singsupp}}
\newcommand{\wave}{\operatorname{wave}}
\newcommand{\ind}{\operatorname{ind}}
\newcommand{\mynull}{\operatorname{null}}
\newcommand{\inj}{\operatorname{inj}}
\newcommand{\arcsinh}{\operatorname{arcsinh}}
\newcommand{\Spec}{\operatorname{Spec}}
\newcommand{\Ind}{\operatorname{Ind}}
\newcommand{\floor}[1]{\left \lfloor #1  \right \rfloor}

\newcommand\restr[2]{{
  \left.\kern-\nulldelimiterspace 
  #1 
  \vphantom{\big|} 
  \right|_{#2} 
  }}


\section{Introduction}

\subsection{Preliminaries} Let $(M,g)$ be a closed $n$-dimensional Riemannian manifold. Consider \textit{the Laplace-Beltrami operator} $\Delta=-\mydiv_g \circ \grad_g$. It is an elliptic self-adjoint operator of second order. Its spectrum is a discrete collection of non-negative eigenvalues with finite multiplicities,

\begin{align*}
0=\lambda_0(g) < \lambda_1(g) \leqslant \lambda_2(g) \leqslant ... \nearrow +\infty.
\end{align*}

We are interested in studying the extremal properties of $\lambda_k(g)$. To this end we consider $\lambda_k(g)$ as a functional on the space $\mathcal{R}(M)$ of Riemannian metrics on $M$,
\begin{align*}
 \lambda_k\colon \mathcal{R}(M) &\to \mathbb{R}_+,\\
g &\mapsto \lambda_k(g).
\end{align*}
However, it turns out that for any positive constant $t>0$ one has
$$
\lambda_k(tg)=\frac{\lambda_k(g)}{t},
$$
which is not convenient for our purposes.
Instead, we consider {\em normalized eigenvalues} defined by

\begin{align*}
\overline{\lambda}_k(M,g)=\lambda_k(g)\Vol(M,g)^{\frac{2}{n}},
\end{align*}
where $\Vol(M,g)$ stands for the volume of the Riemannian manifold $(M,g)$. 

\begin{theorem}[\cite{MR1198600, MR1213857,hassannezhad2011}]
\label{Kor}
One has the following bounds 
\begin{enumerate}[(i)] 
\item If $\dim M=2$, then there exists a constant $C>0$ depending only on the topology of $M$ such that
$$
\overline{\lambda}_k(M,g) \leqslant Ck.
$$ 
\item If $\dim M \geqslant 3$, then the functional $\overline{\lambda}_k(M,g)$ is not bounded from above on the space $\mathcal{R}(M)$.
\item In any dimension there exists a constant $C([g])>0$ depending only on the conformal class $[g]=\{e^{2\omega} g|\,\,\omega\in C^\infty(M)\}$ such that for every metric $\tilde{g} \in [g]$ one has
$$
\overline{\lambda}_k(M, \tilde{g}) \leqslant Ck^{\frac{2}{n}}.
$$
\end{enumerate}
\end{theorem}

\begin{remark}
Theorem~\ref{Kor} holds for any compact manifold with smooth boundary if we replace $\overline{\lambda}_k(M,g)$ by $\overline{\lambda}^N_k(M,g)=\lambda^N_k(g) \Vol(M,g)^{\frac{2}{n}},$ where $\lambda^N_k(g)$ is the $k$-th Neumann eigenvalue of the metric $g$.
\end{remark}

Theorem~\ref{Kor} guarantees that the following quantities are finite
$$
\Lambda_k(M)=\sup_{g\in \mathcal{R}(M)} \overline{\lambda}_k(M,g),
$$
if $\dim M=2$;
$$
\Lambda_k(M,[g])=\sup_{\tilde{g} \in [g]}\overline{\lambda}_k(M, \tilde{g}),
$$
in any dimension. If $\dim M=2$, then we will often use the notation $\Sigma$ instead of $M$. 

The invariant $\Lambda_k(\Sigma)$ has been studied extensively in the last years (see, for example~\cite{MR0292357,MR674407,petrides,petrides2017existence, nadirashvili_preprint, karpukhin2017isoperimetric, nayatani, MR1415764, MR2378458, matthiesen2019handle, girouard-lagace} and references therein). The invariant $\Lambda_k(M,[g])$ is less studied (see, for instance \cite{petrides, el1996premiere, MR2015867, karpukhin-stern}). Below we recall some result which are relevant to our exposition.

\begin{theorem}[\cite{petrides2013rigidity,petrides}]
\label{roman}
Let $(M,g)$ be a closed $n$-dimensional Riemannian manifold not conformally diffeomorphic to the sphere $(\mathbb{S}^n,g_{can}),$ where $g_{can}$ is the standard round metric on $\mathbb{S}^n$. Then one has
$$
\Lambda_1(M,[g]) > \Lambda_1(\mathbb{S}^n, [g_{can}]).
$$
\end{theorem}

\begin{theorem}[\cite{MR2015867}]
For every Riemannian metric $g$ on a closed $n$-dimensional manifold $M$ one has

\begin{align}\label{bruno}
\Lambda_{k}(M, [g]) \geqslant \Lambda_k(\mathbb{S}^n, [g_{can}])
\end{align}
and
\begin{align}\label{colbois}
\Lambda_{k}(M, [g])^{\frac{n}{2}} \geqslant \Lambda_{k-1}(M, [g])^{\frac{n}{2}}+\Lambda_1(\mathbb{S}^n, [g_{can}]).
\end{align}
\end{theorem}

\subsection{Main results}
In this paper we investigate the functional  
\begin{align*}
I_k(M)=\inf_{[g]} \Lambda_k(M,[g]),
\end{align*}
called \textit{the Friedlander-Nadirashvili invariant}. It is a differential invariant depending only on the smooth structure on $M$. 

Let us briefly describe the history of this functional. The invariant $I_1(M)$ was introduced in the paper \cite{MR1717641}, where Friedlander and Nadirashvili proved that for every $n$-dimensional closed manifold $M$ one has
\begin{align*}
I_1(M) \geqslant \overline{\lambda}_1(\mathbb{S}^n, g_{\can}).
\end{align*}
In particular, if $\Sigma$ is a closed surface then
\begin{align*}
I_1(\Sigma) \geqslant 8\pi.
\end{align*}

Inequalities (\ref{bruno}) and (\ref{colbois}) imply that

\begin{align}\label{el soufi}
I_{k}(M) \geqslant \Lambda_k(\mathbb{S}^n, [g_{can}])
\end{align}
and
\begin{align}
\label{IkIk-1}
I_{k}(M)^{\frac{n}{2}} \geqslant I_{k-1}(M)^{\frac{n}{2}}+\Lambda_1(\mathbb{S}^n, [g_{can}]).
\end{align}



We introduce the following notations. Let $\widetilde{\Sigma}_\gamma$ denote an orientable closed surface of genus $\gamma$ and $\Sigma_\gamma$ denote a non-orientable closed surface of genus $\gamma$. Here the genus of a non-orientable closed surface is defined to be the genus of its orientable double cover. Furthermore we set $I_k(\gamma)=I_k(\Sigma_\gamma)$ and $\widetilde{I_k}(\gamma)=I_k(\widetilde{\Sigma}_\gamma)$. In general, we use tilde for anything related to orientable surfaces and do not use it otherwise.

Let us recall known results. Since any two metrics on $\mathbb{S}^2$ or $\mathbb{RP}^2$ are conformally equivalent, 
one has $I_k(0) = \Lambda_k(\mathbb{RP}^2)$ and $\widetilde I_k(0) = \Lambda_k(\mathbb{S}^2)$. According to~\cite{karpukhin2017isoperimetric}, $\Lambda_k(\mathbb{S}^2)=8\pi k$. Similarly, it was proved in~\cite{karpukhin2019index} that $\Lambda_k(\mathbb{RP}^2)=4\pi(2k+1)$. For historical review in research of the invariants $\Lambda_k(\mathbb{S}^2)$ and $\Lambda_k(\mathbb{RP}^2)$ see the survey~\cite{MR4017613}.


In the paper~\cite{MR1717641} Nadirashvili and Friedlander suggested that $I_1(M)=8\pi$ for any closed surface $M$ other than the projective plane. This statement was confirmed in certain cases. In the paper \cite{MR2514484} Girouard proved that $I_1(\mathbb{KL})=I_1(\mathbb{T}^2)=I_1(\mathbb{S}^2)=8\pi$, where $\mathbb{KL}$ is the Klein bottle (see also \cite{MR1415764}). Petrides in the paper \cite{petrides} extended the ideas of Nadirashvili and Girouard and proved that if $M$ is a smooth compact \textit{orientable} surface then $I_1(M)=8\pi$ and the infimum is attained only on the sphere $\mathbb{S}^2$.

The main result of this paper is the following theorem.






\begin{theorem}
\label{disproof}
The following statements hold.
\begin{itemize}
\item[(i)] The Friedlander-Nadirashvili invariants of orientable surfaces satisfy $\widetilde I_k(\gamma)= \widetilde I_k(0)=8\pi k$ for any $\gamma\geqslant 0$. The infimum is attained iff $\gamma=0.$

\item[(ii)] The Friedlander-Nadirahsvili invariants of non-orientable surfaces of odd genus $\gamma\geqslant 1$ satisfy $I_k(\gamma)= \widetilde I_k(0)=8\pi k.$ The infimum is never attained.

\item[(iii)] The Friedlander-Nadirashvili invariants of non-orientable surfaces of even genus $\gamma\geqslant 2$ satisfy
\begin{equation}
\label{even_ineq}
I_k(\gamma) \leqslant I_k(\gamma-2)
\end{equation}
If inequality~\eqref{even_ineq} is strict, then there exists a conformal class $c$ such that $I_{k}(\gamma) = \Lambda_k(\Sigma_\gamma,c)$.
\end{itemize}
\end{theorem}

\begin{corollary} \label{main corollary}
If $\gamma\geqslant 2$ is even, then one has
$$
8\pi k=\widetilde I_k(0) < I_k(\gamma) \leqslant I_k(0)=4\pi(2k+1).
$$
In particular, for $k=1$ one has 
$$
8\pi < I_1(\gamma) \leqslant 12\pi,
$$
for all even $\gamma$.
\end{corollary}

Therefore, Corollary~\ref{main corollary} shows that the statement \textquotedblleft $I_1(M)=8\pi$ unless $M$ is a projective plane\textquotedblright \, suggested by Friedlander and Nadirashvili in~\cite{MR1717641} does not hold for non-orientable surfaces of even genus. 


The main idea in the proof of Theorem~\ref{disproof} is to investigate the behaviour of the quantity $\Lambda_k(M,c_n)$ when the sequence of conformal classes $\{c_n\}$ escapes to infinity in the moduli space of conformal classes on $M$. The precise expression for the limit makes use of Deligne-Mumford compactification. It is stated in Theorem~\ref{conf&conv} and is proved in Section~\ref{appendix}.

As a byproduct of our approach we obtain a result on conformal Neumann eigenvalues that could be of independent interest.
Consider a smooth domain $\Omega$ in $M$. Then we define the following functional

\begin{align*}
\Lambda^{N}_k(\Omega,[g|_{\Omega}]):= \sup_{\tilde{g} \in [g|_{\Omega}]} \overline{\lambda}^{N}_k(\Omega, \tilde{g}),
\end{align*}
where $\overline{\lambda}^{N}_k(\Omega, \tilde{g})=\lambda^N_k(\Omega, \tilde{g})\Vol(\Omega, \tilde{g})^\frac{2}{n}$ and $\lambda^N_k(\Omega, \tilde{g})$ is the $k$-th Neumann eigenvalue of the domain $\Omega$ in the metric $\tilde{g}$. 
In the sequel we often omit the restriction symbol and simply write $\Lambda^{N}_k(\Omega,[g])$. 

\begin{proposition}
\label{main lemma}
Let $(M, g)$ be a compact Riemannian manifold and $\Omega \subset M$ be a smooth domain. Then the following inequality holds,
\begin{align*}
\Lambda_k(M,[g]) \geqslant \Lambda^N_k(\Omega,[g]).
\end{align*}  
\end{proposition} 

\begin{remark}
Similar results for analogs of the Friedlander-Nadirashvili invariants for the \textit{Steklov problem} have been recently obtained by the second named author in the paper~\cite{medved}.
\end{remark}

\subsection{Discussion}
One of the questions that Corollary~\ref{main corollary} leaves unanswered is the exact value of $I_k(\gamma)$ for even $\gamma$. By an analogy with Theorem~\ref{disproof}, (i) and (ii), the following conjecture seems natural.

\begin{conjecture}
\label{evengamma_conj}
 For all even $\gamma$ one has
$$
I_k(\gamma) = I_k(0).
$$
The infimum is attained iff $\gamma=0$.
\end{conjecture}

Another natural question is: why do the quantities $I_k(\gamma)$ take different values for odd and even $\gamma$? Careful analysis of the proof suggests that the answer lies in the theory of {\em cobordisms}. We recall that two closed manifolds $M$ and $M'$ of the same dimension are called {\em cobordant} if there exists a manifold with boundary $W$ such that the boundary $\partial W$ is the disjoint union $M\sqcup M'$. Similarly, $M$ is cobordant to $0$ or null cobordant if there exists $W$ such that $\partial W=M$. One of the basic facts of cobordism theory is that two manifolds are cobordant iff they can be obtained from one another by a sequence of surgeries, see e.g.~\cite{milnor}. In dimension $2$ it implies that attaching a handle does not change the cobordism class. This makes the cobordism theory for surfaces  rather straightforward. Indeed, since $\mathbb{S}^2$ and $\mathbb{KL}$ are obviously cobordant to $0$, one concludes that all orientable surfaces and all non-orientable surfaces of odd genus are cobordant to $0$. By the same token, all non-orientable surfaces of even genus are cobordant to $\mathbb{RP}^2$. The fact that $\mathbb{RP}^2$ is not cobordant to $0$ can be shown using Stiefel-Whitney characteristic classes, see e.g.~\cite{MilnorStasheff}.

Assuming Conjecture~\ref{evengamma_conj}, the quantity $I_k$ is a {\em cobordism invariant} in dimension $2$.
Inequality~\eqref{even_ineq} can be interpreted as monotinicity of $I_k$ with respect to addition of a handle. The monotonicity then can be shown by choosing a degenerate sequence of conformal classes such that the handle collapses in the limit. It turns out that for such sequence the functional $\Lambda_k(M,c)$ is continuous, see Remark~\ref{continuity_remark}. We believe that the same phenomenon occurs in higher dimensions and propose the following extension of Conjecture~\ref{evengamma_conj}.

\begin{conjecture}
\label{bordism_conj}
The quantities $I_k$ are cobordism invariants, i.e. if $M$ is cobordant to $M'$ then $I_k(M)=I_k(M')$. In particular, if $M$ is cobordant to $0$ then $I_k(M)=I_k(\mathbb{S}^{\dim M})=\Lambda_k(\mathbb{S}^{\dim M},[g_{can}])$.
\end{conjecture}

We remark that the cobordism theory has been used by Jammes in the paper~\cite{Jammes2008} to study upper bounds on $I_1$. 
We plan to tackle Conjectures~\ref{evengamma_conj},~\ref{bordism_conj} in the subsequent papers.


\subsection*{Notation} Let us remind the reader that $\widetilde{\Sigma}_\gamma$ denotes an orientable closed surface of genus $\gamma$ and $\Sigma_\gamma$ denotes a non-orientable closed surface of genus $\gamma$,
$I_k(\gamma)=I_k(\Sigma_\gamma)$ and $\widetilde{I_k}(\gamma)=I_k(\widetilde{\Sigma}_\gamma)$. In general, we use tilde to denote anything related to orientable objects. For example, $\pi\colon\widetilde\Sigma_\gamma\to\Sigma_\gamma$ denotes an orientable double cover. Moreover, the notation $\Sigma$ is usually used to denote a non-orientable surface and $\widetilde\Sigma$ is used to denote an orientable surface. If we do not want to specify orientablity of the surface, we denote it by $M$.  

\subsection*{Plan of the paper.} The paper is organized in the following way. 
In Section~\ref{geometry} we provide the geometric background, including hyperbolic surfaces and the convergence on the space of hyperbolic structures on a given surface. There we state the main technical result of the paper -- Theorem~\ref{conf&conv}. In Section~\ref{main theorem proof} we deduce Theorem~\ref{disproof} from Theorem~\ref{conf&conv} and prove Corollary~\ref{main corollary}. Sections~\ref{analysis} and~\ref{appendix} are devoted to proving Theorem~\ref{conf&conv}. In Section~\ref{analysis} we recall necessary facts about Neumann eigenvalues and, finally, in Section~\ref{appendix} we complete the proof.
\subsection*{Acknowledgements.} The authors are grateful to Iosif Polterovich for fruitful discussions
and for his remarks on the initial draft of the manuscript.
 The authors would like to thank Alexandre Girouard for outlining the proof of Proposition \ref{N-cont} and Bruno Colbois for valuable remarks. The authors are thankful to the reviewer for useful remarks and suggestions. During the preparation of this manuscript the first author was supported by Schulich Fellowship. This research is a part of the second author's PhD thesis at the Universit\'e de Montr\'eal under the supervision of Iosif Polterovich.

\medskip

\section{Moduli space of conformal classes}
\label{geometry}
In this section we recall necessary background on the geometry of moduli space of conformal classes on a fixed surface $M$. Even though the contents of this section are mostly classical, we felt inclined to include it in the paper due to the fact that the case of non-orientable surfaces is less known. In our exposition we follow the books \cite{MR1183224, MR1451624}.

The starting point is the uniformization theorem that states that in any conformal class there exists a unique (up to an isometry) metric of constant Gauss curvature and fixed area. Note that the area assumption is unnecessary unless $\chi(M) = 0$ in which case we fix the volume to be equal to $1$. We start with the case $\chi(M)<0$ corresponding to hyperbolic metrics.

\subsection{Orientable hyperbolic surfaces: collar theorem} We start with the definition. 

\begin{definition}
A Riemannian metric $h$ of constant Gaussian curvature $-1$ is called {\em hyperbolic}. A Riemannian surface $(M,h)$ endowed with a hyperbolic metric $h$ is called {\em a hyperbolic surface}. 
\end{definition}

%
%
%
Note that a hyperbolic surface necessarily has negative Euler characteristic.
We recall one of the underlying facts of this theory: the {\em Collar Theorem}. Orientable case is well-known and can be found e.g. in~\cite{MR1183224}.

\begin{definition}
A compact Riemann surface $Y$ of genus 0 with 3 boundary components is called {\em a pair of pants}.
\end{definition}

\begin{theorem}[Collar theorem]\label{Collar theorem}
Let $(\widetilde\Sigma,h)$ be an orientable compact hyperbolic surface of genus $\gamma\geqslant 2$ and let $c_1,c_2,\ldots, c_m$ be pairwise disjoint simple closed geodesics on $(\widetilde\Sigma,h)$. Then the following holds
\begin{enumerate}[(i)]
\item $m \leqslant 3 \gamma-3$.
\item There exist simple closed geodesics $c_{m+1},\ldots,c_{3 \gamma-3}$ which, together with $c_1,\ldots,c_m$, decompose $\widetilde\Sigma$ into pairs of pants.
\item The collars 
\begin{align*}
\mathcal{C}(c_i)=\left\{p\in \widetilde\Sigma~|~ dist(p,c_i) \leqslant w(c_i)\right\}
\end{align*}
of widths 
\begin{align*}
w(c_i)=\frac{\pi}{l(c_i)}\left(\pi-2\arctan\left(\sinh\frac{l(c_i)}{2}\right)\right)
\end{align*}
are pairwise disjoint for $i=1,...,3 \gamma-3$.
\item
Each $\mathcal{C}(c_i)$ is isometric to the cylinder $\left\{(t,\theta)| -w(c_i)<t<w(c_i),\,\theta\in\mathbb{R}/2\pi\mathbb{Z}\right\}$ with the Riemannian metric 

\begin{align*}
\left(\frac{l(c_i)}{2\pi \cos\left(\frac{l(c_i)}{2\pi}t\right)}\right)^2\left(dt^2+d\theta^2\right).
\end{align*}

\end{enumerate}
\end{theorem}


The decomposition of $(\widetilde\Sigma,h)$ into pair of pants is called \textit{the pants decomposition}. We denote it by $\mathcal{P}$. We say that the geodesics $c_1,\ldots,c_{3 \gamma-3}$ form $\mathcal{P}$. 

\subsection{Non-orientable hyperbolic surfaces: collar theorem}
In this section we discuss the case of non-orientable surfaces. Let $(\Sigma,h)$ be a non-orientable hyperbolic surface and let $\pi: \widetilde\Sigma \to \Sigma$ be the orientable double cover. Lifting the metric $h$ to $\widetilde{\Sigma}$ we get an orientable hyperbolic surface $(\widetilde{\Sigma}, \pi^*h)$. If $\tau$ is the involution exchanging the leaves of $\pi$, then $\tau$ is an isometry of $(\widetilde{\Sigma}, \pi^*h)$. In other words, the hyperbolic surface $(\widetilde{\Sigma}, \pi^*h)$ is $\tau$-invariant. 

Let $c$ be a simple closed geodesic on $(\Sigma,h)$. The preimage $\pi^{-1}(c)$ is either a $\tau$-invariant simple closed geodesic $\widetilde c$ on $(\widetilde\Sigma,\pi^*h)$ or a pair $\widetilde c_1$, $\widetilde c_2$ of simple closed geodesics such that $\tau(\widetilde c_1) = \widetilde c_2$. Assume $\pi^{-1}(c) = \widetilde c$. Then $\tau$ acts on the collar $\mathcal C(\widetilde c)$ as an isometry $(t,\theta)\to (-t,\theta+\pi)$. Therefore, the $\pi$-image of the cylinder $\mathcal C(\widetilde c)$ is a M\"obius band $\mathcal C(\widetilde c)/\tau$ around $c$. We refer to this M\"obius band as a collar $\mathcal C(c)$ of $c$ and call $c$ a {\em $1$-sided geodesic}. Now, assume $\pi^{-1}(c) = \widetilde c_1\cup\widetilde c_2$. Then $\tau$ exchanges the collars $\mathcal C(\widetilde c_1)$ and $\mathcal C(\widetilde c_2)$ and their $\pi$-image is a cylinder around $c$. We refer to that cylinder as a collar $\mathcal C(c)$ of $c$ and call $c$ a {\em $2$-sided geodesic}. With that we can state the collar theorem in the non-orientable case.

\begin{theorem}[Collar theorem]\label{Collar theorem nor}
Let $(\Sigma,h)$ be a compact non-orientable hyperbolic surface of genus $\gamma \geqslant 2$ and let $c^1_1,c^1_2,\ldots, c^1_{m_1},c^2_1,\ldots,c^2_{m_2}$ be pairwise disjoint simple closed geodesics on $(\Sigma,h)$, where $c^1_i$ are $1$-sided geodesics and $c^2_j$ are $2$-sided geodesics. Then the following holds
\begin{enumerate}[(i)]
\item $m_1+2m_2 \leqslant 3 \gamma-3$.
\item There exist simple closed geodesics $c^1_{m_1+1},\ldots,c^1_{n_1},c^2_{m_2+1},\ldots,c^2_{n_2}$ which, together with $c^1_1,c^1_2,\ldots, c^1_{m_1},c^2_1,\ldots,c^2_{m_2}$, decompose $\Sigma$ into pairs of pants. Moreover, $c^1_i$ are $1$-sided geodesics, $c^2_j$ are $2$-sided geodesics and $n_1+2n_2 = 3\gamma-3$.
\item The collars 
\begin{align*}
\mathcal{C}(c^\alpha_i)=\left\{p\in \widetilde\Sigma~|~ dist(p,c^\alpha_i) \leqslant w(c^\alpha_i)\right\}
\end{align*}
of widths 
\begin{align*}
w(c^\alpha_i)=\frac{\pi}{\alpha l(c^\alpha_i)}\left(\pi-2\arctan\left(\sinh\frac{\alpha l(c^\alpha_i)}{2}\right)\right)
\end{align*}
are pairwise disjoint for $i=1,...,3 \gamma-3$, $\alpha=1,2$.

\item
Each $\mathcal{C}(c^2_i)$ is isometric to the cylinder $\{(t,\theta)| -w(c^1_i)<t<w(c^1_i),\,\theta\in\mathbb{R}/2\pi\mathbb{Z}\}$ with the Riemannian metric 

\begin{align*}
\left(\frac{l(c^2_i)}{2\pi \cos\left(\frac{l(c^2_i)}{2\pi}t\right)}\right)^2\left(dt^2+d\theta^2\right).
\end{align*}

\item
Each $\mathcal{C}(c^1_i)$ is isometric to the M\"obius band $\{(t,\theta)| -w(c^2_i)<t<w(c^2_i),\, \,\theta\in\mathbb{R}/2\pi\mathbb{Z}\}/\sim$, where $(t,\theta)\sim(-t,\theta+\pi)$, with the Riemannian metric 

\begin{align*}
\left(\frac{l(c^1_i)}{\pi \cos\left(\frac{l(c^1_i)}{\pi}t\right)}\right)^2\left(dt^2+d\theta^2\right).
\end{align*}

\end{enumerate}
\end{theorem}

\begin{proof}
We consider the preimages of all the geodesics on the orientable double cover $\widetilde\Sigma$. We then have a $\tau$-invariant set of simple closed geodesics on $\widetilde\Sigma$. It is proved in the paper~\cite{MR1174602} that every $\tau$-invariant set of simple closed geodesics can be complemented to the $\tau$-invariant set of $3\gamma-3$ simple closed geodesics. This proves $(i)$. The rest follows from the orientable Collar theorem and the discussion above.
\end{proof}

\subsection{Convergence of hyperbolic metrics: orientable case} 
In this section we recall compactness properties of hyperbolic metrics. Our exposition essentially follows the book~\cite{MR1451624}.
Let $\widetilde\Sigma$ be an orientable surface of genus $\gamma\geqslant 2$ and let $\{h_n\}$ be a sequence of hyperbolic metrics on $\widetilde\Sigma$. 
\begin{proposition}[Mumford's compactness theorem]
\label{Mumford}
Assume that the injectivity radii $\inj(\widetilde\Sigma,h_n)$ satisfy $\limsup\limits_{n\to\infty}\inj(\widetilde\Sigma,h_n)>0$. Then there exists a subsequence $\{h_{n_k}\}$, sequence $\{\Phi_k\}$ of smooth automorphisms of $\widetilde\Sigma$ and a hyperbolic metric $h_\infty$ on $\widetilde\Sigma$ such that the sequence of hyperbolic metrics $\{\Phi_k^*h_{n_k}\}$ converges in $C^\infty$-topology to $h_\infty$.
\end{proposition}

We say that a sequence $\{h_n\}$ {\em degenerates} if it does not satisfy the assumptions of Mumford's compactness theorem, i.e. if $\lim\limits_{n\to\infty}\inj(\widetilde\Sigma,h_n)=0$. We now turn to {\em Deligne-Mumford compactification} which allows one to associate a limiting object to a degenerating sequence of hyperbolic metrics. For the remainder of this section assume that $\inj(\widetilde\Sigma,h_n)\to0$.

Under this assumption the thick-thin decomposition implies that for each $n$ there exists a collection $\{c_1^n,\ldots,c_s^n\}$ of disjoint simple closed geodesics in $(\widetilde\Sigma,h_n)$ whose lengths tend to $0$. Moreover, the length of any geodesic in the complement $\widetilde\Sigma_n = \widetilde\Sigma\backslash (c_1^n\cup\ldots\cup c_s^n)$ is bounded from below by a constant independent of $n$. Each $(\widetilde\Sigma_n, h_n)$ is possibly a disconnected hyperbolic surface with geodesic boundary. Up to a choice of a subsequence all components of $\widetilde\Sigma_n$ have the same topological type. We denote by $\widehat{\Sigma_\infty}$ the surface having the same connected components as $\widetilde\Sigma_n$, but with boundary component replaced by marked points. Each sequence $\{c_i^n\}$ gives rise to a pair of marked points $\{p_i,q_i\}$ on $\widehat{\Sigma_\infty}$, $i=1,\ldots, s$. Let us denote by $\Sigma_\infty$ the punctured surface $\widehat{\Sigma_\infty}\backslash\{p_1,q_1,\ldots,p_s,q_s\}$ and by $h_\infty$ the complete hyperbolic metric on $\Sigma_\infty$ with cusps at punctures.

\begin{proposition}[Deligne-Mumford compactification]\label{D-M} 
Let $(\widetilde\Sigma, h_n)$ be a sequence of hyperbolic surfaces such that $\inj(\widetilde\Sigma,h_n)\to 0$. Then up to a choice of subsequence, there exists a sequence of diffeomorphisms $\Psi_n: \Sigma_\infty \to \Sigma_n$ such that the sequence $\{\Psi^*_n h_n\}$ of hyperbolic metrics  converges in $C_{\mathrm{loc}}^\infty$-topology to the complete hyperbolic metric $h_\infty$ on $\Sigma_\infty$. Furthermore, there exists a metric of locally constant curvature $\widehat{h_\infty}$ on $\widehat{\Sigma_\infty}$ such that its restriction to $\Sigma_\infty$ is conformal to $h_\infty$.
\begin{remark}
We say that $\widehat{h_\infty}$ has {\em locally} constant curvature, because $\widehat{\Sigma_\infty}$ could be disconnected and different connected components could have different signs of Euler characteristic. 
\end{remark}

%
%
%
%
%

\end{proposition}

\begin{remark}
\label{pants}
For the general case of hyperbolic surfaces with boundary and cusps see \cite[Proposition 5.1]{MR1451624}.
\end{remark}

When the statement of Proposition~\ref{D-M} holds for the full sequence $\{h_n\}$ we say that $(\widehat{\Sigma_\infty},\widehat{h_\infty})$ is a {\em limiting space} of the sequence $(\Sigma,h_n)$. Similarly, we say that the limit of conformal classes $[h_n]$ is the conformal class $[\widehat{h_\infty}]$ on $\widehat{\Sigma_\infty}$.

\subsection{Convergence of hyperbolic metrics: non-orientable case} To the best of our knowledge, there is no straightforward argument that allows to generalize the contents of the previous section to the non-orientable case. The natural approach is to pass to the double cover to obtain a sequence of hyperbolic $\tau$-invariant metrics and then show that the diffeomorphisms $\Phi_n$ and $\Psi_n$ can be chosen to commute with $\tau$. This approach is taken for example in~\cite[Section 6]{Seppala1991}. In particular, he proves that both Proposition~\ref{Mumford} and~\ref{D-M} hold  for non-orientable surfaces without changes. We remark that the limiting surface $\widehat{\Sigma_\infty}$ can have orientable and non-orientable connected components. 

\begin{remark} \label{conf&conv remark} 
Any conformal class on $\widehat{\Sigma_\infty}$ can be obtained as a limit of conformal classes $[h_n]$ on $\Sigma$. Indeed, consider $\widehat{\Sigma_\infty}$ and a conformal class $[g]$ on it, marked by some metric $g$. Removing points $p_i$ and $q_i$, we then find a hyperbolic metric $h$ in the conformal class $[g_|{_{\widehat{\Sigma_\infty}\setminus \cup^s_{i=1}\{p_i,q_i\}}}]$ to obtain a hyperbolic surface with cusps. Take a pants decomposition of $(\widehat{\Sigma_\infty}\setminus \cup^s_{i=1}\{p_i,q_i\},h)$ and consider singular pants, i.e. pants with cusps instead of boundary. 
For each $\varepsilon>0$ consider a surface with boundary obtained by replacing cusps with boundary components of length $\varepsilon$. Gluing the boundary component corresponding to $p_i$ with the boundary component corresponding to $q_i$ we obtain a hyperbolic surface $(\Sigma,h_\varepsilon)$. From the construction of Deligne-Mumford compactification, it follows that $(\widehat{\Sigma_\infty},[g])$ is the limiting space of $(\Sigma,h_\varepsilon)$ as $\varepsilon\to 0$.
\end{remark}


\subsection{Moduli space in non-negative Euler characteristic}
\label{non-neg}
 Having discussed the hyperbolic surfaces that correspond to the negative Euler characteristic, we proceed to the remaining surfaces: $\mathbb{S}^2$, $\mathbb{RP}^2$, $\mathbb{T}^2$ and $\mathbb{KL}$. In case of $\mathbb{S}^2$ and $\mathbb{RP}^2$ there is a unique conformal class of metrics and as a result the moduli space of conformal classes is a single point. We give an explicit description of the moduli space for $\mathbb{T}^2$ and $\mathbb{KL}$ below.

On the torus $\mathbb{T}^2$ the moduli space of conformal classes is a subset of $\mathbb{R}^2$ given by $\left\{(a,b)|\,a^2+b^2\geqslant 1,\ 0\leqslant a\leqslant 1/2\right\}$. To each $(a,b)$ one can associate a lattice $\Lambda_{a,b}$ in $\mathbb{R}^2$ spanned by vectors $(1,0)$ and $(a,b)$. Then the flat metric $g_{a,b}$ of unit volume on $\mathbb{R}^2/(b^{-\frac{1}{2}}\Lambda_{a,b})$ is a canonical representative of the corresponding conformal class. Let $(a_n,b_n)$ be a sequence of points on the moduli space. Then this sequence has an accumulation point unless $b_n\to+\infty$. Therefore, a degenerating sequence of conformal classes corresponds to $b_n\to+\infty$. Similarly to the hyperbolic case, for the degenerating sequence $(a_n, b_n)$ the injectivity radius $\inj(\mathbb{T}^2,g_{a_n,b_n})\to 0$ as the length of the geodesic $c_n$ corresponding to 
the vector $(b_n^{-\frac{1}{2}},0)$ goes to zero. Moreover, $c_n$ has a cylindrical collar of width $\frac{1}{2}\sqrt{\frac{a_n^2+b_n^2}{b_n}}$ and the limiting space is the sphere $\mathbb{S}^2$ with its unique conformal class.

On the Klein bottle the moduli space of conformal classes is the set of positive real numbers $\mathbb R_{+}$. To each $b>0$ one can associate a group $G_b$ of isometries of $\mathbb{R}^2$ generated by $(x,y)\mapsto(x,y+b^\frac{1}{2})$ and $(x,y)\mapsto (x+b^{-\frac{1}{2}}, -y)$. Then the flat metric $g_b$ of unit volume on $\mathbb{R}^2/G_b$ is a canonical representative of the corresponding conformal class. The sequence of points $\{b_n\}$ has an accumulation point unless $b_n\to 0$ or $b_n\to +\infty$. Therefore, there are two types of degenerating sequences of conformal classes: those corresponding to $b_n\to 0$ and those corresponding to $b_n\to +\infty$. Assume $b_n\to 0$. Then the lengths of geodesics $c_n$ corresponding to the vector $(0,b_n^\frac{1}{2})$ go to zero. Moreover, $c_n$ has a cillindrical collar of width $\frac{1}{2}b_n^{-\frac{1}{2}}$, i.e. $c_n$ is a $2$-sided geodesic, and the limiting space is the sphere $\mathbb{S}^2$ with its unique conformal class. Assume $b_n\to+\infty$. Then the lengths of geodesics $d_n$ corresponding to the vector $(b_n^{-\frac{1}{2}},0)$ go to zero. Moreover, $d_n$ has a M\"obius band collar of width $\frac{1}{2}b_n^{\frac{1}{2}}$, i.e. $d_n$ is a $1$-sided geodesic, and the limiting space is the projective plane $\mathbb{RP}^2$ with its unique conformal class. Either way, $\inj(\mathbb{KL},g_{b_n})\to 0$.

\subsection{Degenerating conformal classes}

From now on we no longer use $c$ to denote geodesics and reserve the letter $c$ to denote conformal classes.
\begin{definition}
Let $M$ be a surface and let $\{c_n\}$ be a sequence of conformal classes on $M$. Let $h_n\in c_n$ be a canonical representative, i.e. $h$ is hyperbolic if $\chi(M)<0$ and $h$ is flat of unit volume if $\chi(M)=0$. We say that $c_n$ {\em degenerates} if $\inj(M,h_n)\to 0$. Furthermore, if $(M,h_n)\to(\widehat{M_\infty},\widehat{h_\infty})$ in the sense of Proposition~\ref{D-M} (if $\chi(M)<0$) or in the sense of Section~\ref{non-neg} (if $\chi(M)=0$), then we say that $c_n$ converges to $c_\infty=[\widehat{h_\infty}]$.
\end{definition}

In~\cite{cianci2019branched} it is shown that if the sequence $c_n$ does not degenerate and converges to $c$ then one has 
$\Lambda_k(M,c_n)\to\Lambda_k(M,c)$.
The main technical result of the present paper establishes the value of the limit of $\Lambda_k(M,c_n)$ when the sequence of conformal classes $c_n$ degenerates.

\begin{theorem}\label{conf&conv}
Let $M$ be a closed compact (orientable or non-orientable) surface and let $c_n\to c_\infty$ be a degenerating sequence of conformal classes. 
Suppose that $\widetilde s$ $2$-sided and $s$ $1$-sided geodesics collapse, so that the surface $\widehat{M_{\infty}}$ has $\widetilde m$ orientable components $\widetilde \Sigma_{\widetilde\gamma_{i}}$ of genus $\widetilde\gamma_i$, $i=1,\ldots,\widetilde m$ and $m$ non-orientable components $\Sigma_{\gamma_{j}}$ of genus $\gamma_j$, $j=1,\ldots, m$. Then one has

\begin{equation}
\label{deg_limit}
\begin{split}
&\lim_{n \to \infty} \Lambda_k (M, c_n) =\\ &\max \Big(\sum^{\widetilde m}_{i=1} \Lambda_{\widetilde k_i}(\widetilde\Sigma_{\widetilde\gamma_{i}}, c_\infty)+\sum^{m}_{i=1} \Lambda_{k_{i}}(\Sigma_{\gamma_{i}},c_\infty) + \sum_{i=1}^{\widetilde s}\Lambda_{\widetilde r_i}(\mathbb{S}^2) + \sum_{i=1}^s\Lambda_{r_i}(\mathbb{RP}^2) \Big),
\end{split}
\end{equation}
where the maximum is taken over all possible combinations of indices such that
 $$
 \sum_{i=1}^m k_i + \sum_{i=1}^{\widetilde m}\widetilde k_i + \sum_{i=1}^s r_i + \sum_{i=1}^{\widetilde s}\widetilde r_i = k.
 $$

%

\end{theorem}
\begin{remark}
\label{continuity_remark}
We remark that inequality~\eqref{colbois} implies that the terms $\Lambda_{\widetilde r_i}(\mathbb{S}^2)=8\pi \widetilde r_i$ in the r.h.s of~\eqref{deg_limit} can be absorbed into the other terms. This fact together with Lemma~\ref{disconnected} below allows us to formulate equality~\eqref{deg_limit} in a way that resembles continuity property,
\begin{align*}
&\lim_{n \to \infty} \Lambda_k (M, c_n) =\max_{k-s\leqslant k'\leqslant k, k'\geqslant 0} \Big(\Lambda_{k'}(\widehat{M_\infty},c_\infty) + 12\pi(k-k')\Big),
\end{align*}
where we have used the fact that $\Lambda_{r}(\mathbb{RP}^2)=4\pi(2r+1)$. As a result, the functional $\Lambda_k(M,c_n)$ is not necessarily continuous for degenerating sequences of conformal classes as long as at least a single $1$-sided geodesic collapses.
\end{remark}

\begin{remark}
A result similar to Theorem~\ref{conf&conv} for the Steklov problem has been recently obtained in the paper~\cite{medved} (see Theorem 1.2).
\end{remark}

The proof of Theorem~\ref{conf&conv} is rather technical. We postpone it until Section~\ref{appendix}.

\subsection{Topology of the limiting space} The following purely topological lemma describes the relation between the genera of $M$ and $\widehat{M_\infty}$.

\begin{lemma} 
\label{topology_eq}
\begin{itemize}
\item[(i)] Let $c_n\to c_\infty$ be a degenerating sequence of conformal classes on $\widetilde\Sigma_{\widetilde\gamma}$.
Suppose that $\widetilde s$ geodesics collapse, so that the surface $\widehat{\Sigma_{\gamma,\infty}}$ has $\widetilde m$ components $\widetilde \Sigma_{\widetilde\gamma_{i}}$ of genus $\widetilde\gamma_i$, $i=1,\ldots,\widetilde m$. Then one has
\begin{align}\label{or}
\widetilde\gamma = \widetilde s + |\widetilde\Gamma|-\widetilde m + 1
\end{align}
where $\widetilde\Gamma=\{\widetilde\gamma_{1},\ldots,\widetilde\gamma_{\widetilde m}\}$, $|\widetilde\Gamma| = \sum_{i=1}^m\widetilde\gamma_i$.

\item[(ii)] Let $c_n\to c_\infty$ be a degenerating sequence of conformal classes on $\Sigma_\gamma$. Suppose that $\widetilde s$ $2$-sided and $s$ $1$-sided geodesics collapse, so that the surface $\widehat{\Sigma_{\gamma,\infty}}$ has $\widetilde m$ orientable components $\widetilde \Sigma_{\widetilde\gamma_{i}}$ of genus $\widetilde\gamma_i$, $i=1,\ldots,\widetilde m$ and $m$ non-orientable components $\Sigma_{\gamma_{j}}$ of genus $\gamma_j$, $j=1,\ldots, m$. Then one has

\begin{align}\label{nonor}
\gamma=2(\widetilde s+|\widetilde\Gamma|-\widetilde m)+s+|\Gamma|-m+1,
\end{align}
where $\widetilde\Gamma=\{\widetilde\gamma_{1},\ldots,\widetilde\gamma_{\widetilde m}\}, \Gamma=\{\gamma_{1},\ldots,\gamma_{m}\}$, $|\widetilde\Gamma| = \sum_{i=1}^{\widetilde m}\widetilde\gamma_i$ and $|\Gamma| = \sum_{i=1}^m\gamma_i$.
\end{itemize}
\end{lemma} 

\begin{proof}

$(i)$ The surface $\widetilde\Sigma_{\widetilde\gamma}$ is obtained from components $\widetilde\Sigma_{\widetilde\gamma_i}$ by joining them with $\widetilde s$ cylinders. Recall that Mayer--Vietoris sequence implies that if $M=M_1\cup M_2$, then the Euler characteristics satisfy the following relation, $\chi(M) = \chi(M_1)+\chi(M_2)-\chi(M_1\cap M_2)$. We apply this formula to $M_1$ -- disjoint union of $\widetilde\Sigma_{\widetilde\gamma_i}$ with $\widetilde s_i$ holes, $M_2$ -- disjoint union of $\tilde s$ cylinders, $M$ is $M_1$ and $M_2$ glued by a common boundary.  Since $\sum \widetilde s_i = 2\widetilde s$, one has
$$
2-2\widetilde\gamma=\chi(\widetilde\Sigma_\gamma) = \sum_j(2-2\widetilde\gamma_j - \widetilde s_j) = 2\widetilde m - 2|\widetilde\Gamma| - 2\widetilde s. 
$$
Rearranging the terms yields~\eqref{or}.

$(ii)$ Non-orientable case follows from the orientable case by passing to the double cover: $2$-sided collapsing geodesics lift to a pair of collapsing geodesics; $1$-sided collapsing geodesics lift to a single collapsing geodesic; orientable components $\widetilde\Sigma_{\widetilde\gamma_i}$ lift to two copies of itself and non-orientable components $\Sigma_{\gamma_j}$ lift to its orientable double cover $\widetilde\Sigma_{\gamma_j}.$

\end{proof}

\begin{corollary}
\label{topology_even_genus}
In notations of Lemma~\ref{topology_eq}(ii) assume $\gamma$ is even. Then either $s\ne 0$ or $\gamma_i$ is even for some $i$. 
\end{corollary}
\begin{proof}
By Lemma~\ref{topology_eq} one has
$$
\gamma=2(\widetilde s+|\widetilde\Gamma|-\widetilde m)+s+|\Gamma|-m+1.
$$
If $\gamma_i$ is odd for all $i$, then $|\Gamma|-m$ is even. Since $\gamma$ is even, this implies $s+1$ is even, i.e. $s\ne 0$.
\end{proof}

We conclude this section with the following observation.

\begin{lemma}\label{metric sequences}
\begin{itemize}
\item[(i)] On $\widetilde\Sigma_{\widetilde\gamma}$
there exists a degenerating sequence of conformal classes $\{c_n\}$ such that the limiting space 
$\widehat{\Sigma_\infty}$ is a union of spheres.
\item[(ii)] Let $\Sigma_\gamma$ be a non-orientable surface of odd genus $\gamma$. Then there exists a degenerating sequence of conformal classes $\{c_n\}$ such that all the collapsing geodesics are $2$-sided and the limiting space $\widehat{\Sigma_\infty}$ is a union of spheres.
\item[(iii)] Let $\Sigma_\gamma$ be a non-orientable surface of even genus $\gamma\geqslant 2$. Then for any even $\gamma'<\gamma$ and any conformal class $c$ on $\Sigma_{\gamma'}$ there exists a degenerating sequence of conformal classes $\{c_n\}$ such that all the collapsing geodesics are $2$-sided and the limiting space $\widehat{\Sigma_\infty}$ is a union of spheres and a surface $\Sigma_{\gamma'}$ equipped with a conformal class $c$.
\end{itemize}
\end{lemma}
\begin{proof} 
From the discussion in Section~\ref{non-neg} this lemma is obvious in the non-negative Euler characteristic. In the remainder of the proof we focus on hyperbolic surfaces. 

Consider a hyperbolic orientable surface $(\widetilde\Sigma_\gamma,h)$ of genus $\gamma$. Given a pants decomposition $\mathcal P$ of $(\widetilde\Sigma_\gamma,h)$ (see e.g. Figure~$1$), one can construct a new hyperbolic metric $h_\varepsilon$ by replacing all pants in $\mathcal P$ by pants whose boundaries are scaled by $\varepsilon$. Sending $\varepsilon$ to $0$ gives the required sequence.

To show (ii) we refer to Figure~$1$. It pictures a particular pants decomposition of the orientable double cover with the involution given by a reflection with respect to the center point. We see that the involution exchanges pairs of geodesics, i.e. all geodesics in the pants decomposition are $2$-sided. Sending their lengths to $0$ provides the required sequence. 

\begin{figure}
  \centering
  \def\svgwidth{\columnwidth}
  \input{A.tex}
  \begin{minipage}{127mm}
    \footnotesize
    \emph{
    Figure 1: involution-invariant pants decomposition for an orientable double cover of a non-orientable surface of odd genus. The involution is given by the reflection with respect to the center point. Sending the lengths of all geodesics in the decomposition to zero provides the sequence required to prove (ii).
    }
   \end{minipage}
\end{figure}

To show (iii) we refer to Figure~$2$. Once again, it pictures a particular pants decomposition of the orientable double cover with the involution given by a reflection with respect to the center point. The numbers on the bottom refer to the number of handles in the marked interval. The only $1$-sided geodesic is the one corresponding to the central blue geodesic, i.e. all red geodesics project onto $2$-sided geodesics. Sending the lengths of all {\em red} geodesics in the the decomposition to zero provides the sequence satisfying topological requirements of (iii). Moreover, by Remark~\ref{conf&conv remark} any conformal class on the limiting space can be achieved, therefore, the proof of the lemma is complete.

\begin{figure}
  \centering
  \def\svgwidth{\columnwidth}
  \input{B.tex}
  \begin{minipage}{127mm}
    \footnotesize
    \emph{
	Figure 2: involution-invariant pants decomposition for an orientable double cover of a non-orientable surface of even genus. The involution is given by the reflection with respect to the center point. Sending lengths of all {\em red} geodesics in the decomposition to zero provides the sequence required to prove (iii).
    }
   \end{minipage}
\end{figure}
\end{proof}

\section{Proof of Theorem \ref{disproof}}
\label{main theorem proof}

\subsection{Case (i)} Let $\widetilde\Sigma_{\widetilde\gamma}$ be an orientable surface of genus $\widetilde\gamma$. By Lemma~\ref{metric sequences} there exists a sequence of conformal classes $c_n$ such that the limiting space $\widehat{\Sigma_\infty}$ is a union of spheres. Since in the orientable case all geodesics are $2$-sided, Theorem~\ref{conf&conv} implies
$$
\lim_{n\to\infty}\Lambda_k(\widetilde\Sigma_\gamma,c_n) = \max_{\sum k_j=k}\sum\Lambda_{k_j}(\mathbb{S}^2).
$$
We recall that by results of~\cite{karpukhin2017isoperimetric} one has $\Lambda_k(\mathbb{S}^2) = 8\pi k$. Therefore,
$$
\widetilde I_k(\widetilde\gamma)\leqslant \lim_{n\to\infty}\Lambda_k(\widetilde\Sigma_{\widetilde\gamma},c_n) = 8\pi k.
$$
At the same time, by~\eqref{el soufi} one has $\widetilde I_k(\widetilde\gamma)\geqslant 8\pi k.$

\subsection{Case (ii)} Let $\Sigma_\gamma$ be a non-orientable surface of odd genus $\gamma$. By Lemma~\ref{metric sequences} there exists a sequence of conformal classes $c_n$ such that all collapsing geodesics are $2$-sided and the limiting space of $\widehat{\Sigma_\infty}$ is a union of spheres. Then the same argument as in Case (i) yields $I_k(\gamma) = 8\pi k.$

\subsection{Case (iii)} Let $\Sigma_\gamma$ be a non-orientable surface of even genus $\gamma$. 
By Corollary~\ref{topology_even_genus}, for any degenerate sequence $c_n$ of conformal classes on $\Sigma_\gamma$ either the limiting space $\widehat{\Sigma_{\infty}}$ contains non-orientable components of even genus or there exist $1$-sided collapsing geodesics.
We denote by $\Sigma_{\gamma'_i}$ the non-orientable components of $\widehat{\Sigma_{\infty}}$ of even genus $\gamma'_i$ as well as projective planes with $\gamma'_i=0$ for each collapsing $1$-sided geodesic. Let $M_j$ be the remaining components (orientable or non-orientable of odd genus). Then, Theorem~\ref{conf&conv} yields
$$
I_k(\gamma)\leqslant\lim_{n\to\infty}\Lambda_k(\Sigma_\gamma,c_n) = \max\left(\sum\Lambda_{k_j}(M_j,c_\infty) + \sum\Lambda_{k'_j}(\Sigma_{\gamma'_i},c_\infty) \right).
$$  

By Remark~\ref{conf&conv remark} one has that the conformal classes on the right hand side of the previous inequality range over all possible combinations of conformal classes on connected components of the limiting space. Therefore, taking the infimum over all possible degenerating sequences $\{c_n\}$ yields
\begin{equation}
\label{caseiiia}
I_k(\gamma)\leqslant \min_{\{\gamma'_j\},\,\{M_j\}}\max\left(\sum I_{k_j}(M_j) + \sum I_{k'_j}(\gamma'_j)\right),
\end{equation}
where the minimum is taken over all possible topological types of the limiting space. Let $K'$ be the set of indices $k'_j$ and $|K'|$ denote the sum of all $k_j'$. Similarly, let $\Gamma'$ be the set of genera $\gamma'_j$ and $|\Gamma'|$ be the sum of $\gamma'_j.$
Taking into account cases (i) and (ii) proved above, inequality~\eqref{caseiiia} implies
\begin{equation}
\label{caseiiib}
I_k(\gamma)\leqslant\min_{\Gamma'}\max_{K',\,|K'|\leqslant k}\left(\sum I_{k'_j}(\gamma'_j)+8\pi(k-|K'|)\right),
\end{equation}
where the minimum is taken over all possible $\Gamma'$ the limiting space can have.
Lemma~\ref{metric sequences} implies that for all even $\gamma'<\gamma$ the sets $\Gamma' = \{\gamma'\}$ are possible. We claim that the minimum is actually attained on these one element sets.
Indeed, assume $\Gamma'$ contains two elements $\gamma'_1$ and $\gamma'_2$, then by inequality~\eqref{el soufi} for any $k_1,\,k_2$ one has $I_{k_1}(\gamma'_1) + I_{k_2}(\gamma'_2)\geqslant I_{k_1}(\gamma'_1) + 8\pi k_2$.  
Thus, inequality~\eqref{caseiiib} becomes
$$
I_k(\gamma)\leqslant \min_{\gamma'<\gamma}\max_{k'\leqslant k} (I_{k'}(\gamma') + 8\pi(k-k')).
$$
Furthermore, by inequality~\eqref{el soufi} the maximum is achieved when $k'=k$. Therefore,
$$
I_k(\gamma) = \min_{\gamma'<\gamma} I_k(\gamma').
$$
Finally, since this inequality holds for all $\gamma$ it is equivalent to having $I_k(\gamma)\leqslant I_k(\gamma-2)$ for all even $\gamma>0$.

If inequality~\eqref{even_ineq} is strict, then the minimizing sequence of conformal classes can not be degenerate. Therefore, it has to have a genuine conformal class on $\Sigma_\gamma$ as an accumulation point. By continuity of $\Lambda_k(\Sigma_\gamma,c)$ in $c$, see~\cite{cianci2019branched}, one has $I_k(\gamma)=\Lambda_k(\Sigma_\gamma,c)$.

\subsection{Proof of Corollary \ref{main corollary}}
\label{main corollary proof}

We start with the following proposition.

\begin{proposition}
\label{petrides_cor}
Let $(M,g)$ be a closed Riemannian surface not diffeomorphic to the sphere $\mathbb{S}^2$. Then one has
$$
\Lambda_k(M,[g]) > \Lambda_k(\mathbb{S}^2)=8\pi k.
$$
\end{proposition}
\begin{proof}
It is proved in~\cite{karpukhin2017isoperimetric} that $\Lambda_k(\mathbb{S}^2)=8\pi k$.
Then, a combination of Theorem~\ref{roman} and inequality~\eqref{colbois} yields
$$
\Lambda_k(M,[g])\geqslant \Lambda_1(M,[g]) + 8\pi(k-1)>8\pi k = \Lambda_k(\mathbb{S}^2).
$$
\end{proof}

Now we are ready to prove the inequality $I_k(\gamma)>8\pi k$. Assume the contrary. Since $I_k(\gamma)\geqslant 8\pi k$, it implies that $I_k(\gamma)=8\pi k$. At the same time $I_k(0)>8\pi k$. Therefore, there exists an even $\gamma'$, $2\leqslant\gamma'\leqslant \gamma$ such that $I_k(\gamma)=\ldots =I_k(\gamma')< I_k(\gamma'-2)$. As a result, there exists a conformal class $c$ on $\Sigma_{\gamma'}$ such that $I_k(\gamma)=I_k(\gamma')=\Lambda_k(\Sigma_{\gamma'},c)>8\pi k$ by Proposition~\ref{petrides_cor}.

\section{Neumann eigenvalues}\label{analysis}

In this section we recall some results on conformal Neumann eigenvalues. The results of the present section are used repeatedly in Section~\ref{appendix}.
\subsection{Convergence of Neumann spectrum}

\begin{lemma}\label{Neumann conv}
Let $(M, g)$ be a closed compact Riemannian manifold. Consider a finite collection $\{ B_\epsilon(p_i) \}_{i=1}^l$ of geodesic balls of radius $\epsilon$ centred at some points $p_1,\ldots,p_l \in M$. Then for all $k\geqslant 0$ the Neumann eigenvalues $\lambda^N_i(M \setminus \cup^l_{i=1}B_\epsilon(p_i), g)$ converge to the eigenvalues $\lambda_i(M,g)$ as $\epsilon \to 0$.

\end{lemma}

For the proof we refer the reader to the paper \cite[Theorem 2]{anne1987spectre}. 


Next, we recall the following statement.

\begin{proposition}\label{N-cont}
Let $M$ be a closed $n$-dimensional manifold and $\Omega\subset M$ be a smooth domain. Assume the sequence of Riemannian metrics $g_m$ on $M$ converges in $C^\infty$-topology to the metric $g$. Then $\Lambda_k(M,[g_m])\to\Lambda_k(M,[g])$. Similarly, if $h_m|_{\overline\Omega}$ converge to $g|_{\overline\Omega}$ in $C^\infty$-topology, then $\Lambda^N_k(\Omega,[h_m|_{\overline\Omega}])\to\Lambda^N_k(\Omega,[g|_{\overline\Omega}])$. 
\end{proposition}
\begin{proof}
We show the statement for closed manifolds. The case of domains is treated in the exactly same way.

Let $\varepsilon>0$. Then for large enough $m$ one has 
$$
\frac{1}{(1+\varepsilon)^2} f g_m(v,v) \leqslant f g(v,v) \leqslant (1+\varepsilon)^2 f g_m(v,v),\quad \forall v \in \Gamma(TM\setminus\{0\}),
$$
where $f$ is any positive smooth function on $M$.
Then by~\cite[Proposition 3.3]{MR656119} one has
$$
\frac{1}{(1+\varepsilon)^{2(n-1)}}\lambda_k(M,fg_m)\leqslant\lambda_k(M,fg)\leqslant (1+\varepsilon)^{2(n-1)}\lambda_k(M,fg_m).
$$
At the same time 
$$
\frac{1}{(1+\varepsilon)^{n}}\Vol(M,fg_m)\leqslant\Vol(M,fg)\leqslant (1+\varepsilon)^{n}\Vol(M,fg_m).
$$
As a result,
$$
\frac{1}{(1+\varepsilon)^{2n}}\bar\lambda_k(M,fg_m)\leqslant\bar\lambda_k(M,fg)\leqslant (1+\varepsilon)^{2n}\bar\lambda_k(M,fg_m).
$$
Taking the supremum over all $f$ yields
$$
\frac{1}{(1+\varepsilon)^{2n}}\Lambda_k(M,[g_m])\leqslant\Lambda_k(M,[g])\leqslant (1+\varepsilon)^{2n}\Lambda_k(M,[g_m]).
$$
Since it holds for any $\varepsilon>0$ the proof is complete.
\end{proof}
\subsection{Discontinuous metrics}\label{main lemma proof}

Let $(M,g)$ be a closed Riemannian manifold of dimension $n$.
Consider a set of pairwise disjoint smooth domains $\{\Omega_i\}^s_{i=1}$ in $M$ such that $M=\bigcup^s_{i=1} \overline\Omega_i$. 
Let us consider a class of discontinuous metrics on $M$ defined as $\rho g\in[g]$, where $\rho|_{\Omega_i} = \rho_i\in C^\infty(\overline\Omega_i)$ are positive.
The space of such functions $\rho$ will be denoted as $C^{\infty}_+(M,\{\Omega_i\})$. If we do not require the components to be positive, we omit the subscript $+$.

The metric $\rho g$ is not smooth. The spectrum of the Laplacian $\Delta_{\rho g}$ is defined as the set of critical values of the Rayleigh quotient
$$
R_{\rho g}[\varphi]=\frac{\displaystyle\int_M \rho^{\frac{n-2}{2}} | \nabla_g \varphi|^2_g dv_g}{\displaystyle\int_M \rho^{\frac{n}{2}}  \varphi^2 dv_g}.
$$ 
Let $n_i$ be outward pointing normal vector for $(\overline\Omega_i,\rho_i g)$. Then an eigenfunction $u$ corresponding to the eigenvalue $\lambda$ is continuous across $\partial\Omega$ and satisfies the following system
\begin{align*}
 \begin{cases}
 \Delta_{\rho g} u = \lambda u \quad &\mathrm{on}\quad \bigcup\Omega_i,\\
\rho_i^\frac{n-1}{2}\dfrac{\partial u}{\partial n_i} + \rho_j^{\frac{n-1}{2}}\dfrac{\partial u}{\partial n_j} = 0\quad&\mathrm{on}\quad\overline\Omega_i\cap\overline\Omega_j.
 \end{cases}
\end{align*}
Let $C_b(M,\{\Omega_i\})\subset C^0(M)$ be a subspace of $C^{\infty}(M,\{\Omega_i\})$ consisting of functions $v$ satisfying the above boundary condition for eigenfunctions. Then 
$$
\lambda_k(M,\rho g) = \inf_{E_{k+1}} \sup_{\varphi\in E_{k+1}} R_{\rho g}[\varphi],
$$
where $E_{k+1}$ ranges over all $(k+1)$-dimensional subspaces of $C_b(M,\{\Omega_i\})$.

Let us introduce the following notation
\begin{align*}
\Lambda_k(M,\{\Omega_i\},[g])=\sup \{ \bar\lambda_k(\rho g)~\vert~ \rho \in  C^{\infty}_+(M,\{\Omega_i\})\},
\end{align*}
where $\bar\lambda_k(\rho g)$ is the normalized $k$-th eigenvalue given by
$$
\bar\lambda_k(\rho g) = \lambda_k(\rho g) ||\rho||_{L^\frac{n}{2}(M,g)}.
$$


\begin{lemma}\label{identity}
Let $(M,g)$ be a Riemannian manifold of dimension $n$. Consider a set of pairwise disjoint smooth domains $\{\Omega_i\}^s_{i=1}$ in $M$ such that $M=\bigcup^s_{i=1} \overline\Omega_i$. Then one has
\begin{align*}
\Lambda_k(M,\{\Omega_i\},[g])=\Lambda_k(M,[g])
\end{align*}
\end{lemma} 
\begin{proof}
In the paper~\cite[Lemma 2]{MR1717641} this lemma is proved for $k=1$. The proof carries over to the case of arbitrary $k$. The only change is to redefine the set $S$ from the original proof to be

\begin{align*}
S=\{u\in H^1(M,g) \vert~ u \perp_{L^2(M, \rho g)} E_0,..., E_{k-1}, \int_M \rho^{\frac{n}{2}}u^2 dv_g =1\},
\end{align*}
where $E_k$ is the eigenspace corresponding to the $k$-th eigenvalue of the metric $\rho g$. We refer the reader to \cite{MR1717641} for details.

\end{proof}

\begin{lemma}
\label{identity2}
Let $(M,g)$ be a closed Riemannian manifold of dimension $n$. Consider a set of pairwise disjoint smooth domains $\{\Omega_i\}^s_{i=1}$ in $M$ such that $M=\bigcup^s_{i=1} \overline\Omega_i$. Let $(\Omega,h) = \sqcup(\overline\Omega_i,g|_{\overline\Omega_i})$. Then for all $k\geqslant 0$ one has
$$
\Lambda_k(M,[g]) \geqslant \Lambda^N_k(\Omega,[h]).
$$
If $(M,g)$ is compact with non-empty boundary with $(\Omega, h)$ as above, then 
$$
\Lambda^N_k(M,[g]) \geqslant \Lambda^N_k(\Omega,[h]).
$$
\end{lemma}

\begin{proof}
The proof is a combination of a classical Dirichlet-Neumann bracketing argument and Lemma~\ref{identity}. It remains the same whether $M$ has boundary or not. Below, we assume that $M$ is closed.

Let $h^m\in[h]$ be a maximizing sequence of metrics for $\Lambda^N_k(\Omega,[h])$. Let $g^m\in[g]$ be a discontinuous metric on $M$ defined as $g|_{\Omega_i} = h_i$. Since the space of test functions for the Neumann eigenvalues $C^\infty(M,\{\Omega_i\})$ is larger than $C_b(M,\{\Omega_i\})$, the variational characterization of eigenvalues implies that for all $k$ one has $\lambda_k(M,g^m)\geqslant\lambda^N(\Omega,h^m)$. Taking the limit and using the fact that $\Vol(\Omega,h^m)=\Vol(M,g^m)$ yields
$$
\Lambda_k(M,\{\Omega_i\},[g])\geqslant\Lambda_k^N(\Omega,[h]).
$$
An application of Lemma~\ref{identity} completes the proof.
\end{proof}

\subsection{Neumann spectrum of a subdomain.} 

The present section is devoted to the proof of Proposition~\ref{main lemma}. The idea is to introduce a conformal factor that vanishes outside $\Omega$. However, the conformal factors are not allowed to be equal to $0$. To circumvent this difficulty one has to go through an approximation procedure which is carried out below.

Let us first remind the statement of Proposition~\ref{main lemma}. We state it in a slightly more general way.
\begin{proposition}
\label{subdomain}
Let $(M,g)$ be a closed Riemannian manifold, $\Omega\subset M$ is a smooth subdomain. Then for all $k$ one has
$$
\Lambda_k(M,[g])\geqslant \Lambda^N_k(\Omega,[g|_{\overline\Omega}]).
$$
If $M$ is compact with non-empty boundary and $\Omega\subset M$ is a smooth domain, then for all $k$
$$
\Lambda^N_k(M,[g])\geqslant \Lambda^N_k(\Omega,[g|_{\overline\Omega}]).
$$
\end{proposition}
The proof of the boundary case is identical to the closed case. For that reason we only present the closed case below. 

We introduce the conformal factor $\rho_\delta$, so that $\rho_\delta|_{\Omega}\equiv 1$ and $\rho_\delta|_{M\setminus\Omega}\equiv\delta$.
\begin{lemma}\label{liminf}
 One has $$\liminf_{\delta \to 0}\lambda_k(\rho_\delta g)\geqslant \lambda^N_k(\Omega,g),$$ where $\lambda^N_k(\Omega,g)$ is the $k$-th Neumann eigenvalue of the domain $(\Omega,g)$. 
\end{lemma}

Let us first show how to deduce Proposition~\ref{subdomain} from Lemma~\ref{liminf}.
\begin{proof}[Proof of Proposition~\ref{subdomain}]

Let $\{h_i~ |~ h_i \in [g|_{\Omega}]\}$ be a maximizing sequence of metrics for the functional  $\Lambda^{N}_k(\Omega,[g])$, i.e. 

\begin{align*}
\lim_{i\to\infty}\bar{\lambda}^{N}_k(\Omega, h_i)=\Lambda^{N}_k(\Omega, [g])
\end{align*}

Let $h_i=f_i g|_{\Omega}$, where $f_i \in C^\infty_+(\bar\Omega)$. We define the metric $\widetilde{h_i}=\widetilde{f_i} g$ on $M$, where $\widetilde{f_i}$ is any positive continuation of the function $f_i$ into $\Omega^c$. Then we consider the metric $\rho_\delta \widetilde{h_i}$, where as before
 \begin{align*}
\rho_\delta=
 \begin{cases}
 1&\text{in $\Omega$},\\
 \delta&\text{in $M\setminus\Omega$}.
 \end{cases}
\end{align*}
By Lemma \ref{liminf} we then have

\begin{align*}
\liminf_{\delta \to 0} \lambda_k(\rho_\delta \widetilde{h_i}) \geqslant \lambda^N_k(\Omega, h_i). 
\end{align*}
At the same time, $\Vol(M,\rho_\delta\widetilde{h_i})\to \Vol(\Omega,h_i)$.
Then, by Lemma \ref{identity} one obtains

\begin{align*}
\Lambda_k(M, [g])=\Lambda_k(M,\{\Omega,M\setminus\Omega\},[g]) \geqslant \liminf_{\delta\to 0}\bar\lambda_k(\rho_\delta \widetilde{h_i})\geqslant\bar\lambda^N_k(\Omega, h_i).
\end{align*}
Taking the limit as $i \to \infty$ yields,

\begin{align*}
\Lambda_k(M, [g]) \geqslant \Lambda^{N}_k(\Omega, [g]).
\end{align*}

\end{proof}

\begin{proof}[Proof of Lemma~\ref{liminf}]
The proof below is essentially the proof in~\cite[Section 2, Step 2, Step 3]{enciso2015} with details added. We denote $M\setminus\Omega$ by $\Omega^c$.
Let
$$
\mathcal{H}_1:=\{\varphi \in H^1(M,g) ~|~(\Delta \varphi)_{|_{\Omega^c}}=0\}
$$
and
$$
\mathcal{H}_2:=\{\varphi \in H^1(M,g) ~|~\varphi \in H^1_0(\Omega^c,g),~\varphi_{|_\Omega}=0\}.
$$
{\bf Claim 1.} One has the following decomposition of $H^1(M,g)$
$$
H^1(M,g)=\mathcal{H}_1 \oplus \mathcal{H}_2
$$
into the sum of closed subspaces. Moreover for any $\delta>0$ one has
$$
\int_M\langle\nabla \varphi, \nabla \psi \rangle_{\rho_\delta g} dv_{\rho_\delta g}=0, \forall \varphi \in \mathcal{H}_1, \psi \in \mathcal{H}_2,
$$
where as before $\rho_\delta|_{\Omega}\equiv 1$ and $\rho_\delta|_{M\setminus\Omega}\equiv\delta.$
\begin{proof} 
Since $H^1_0(\Omega,g)$ is complete we immediately conclude that $\mathcal{H}_2$ is a closed subspace of $H^1(M,g)$. 

We show that the space $\mathcal{H}_1$ is also closed. Consider the mapping:
$$
T\colon H^1(M,g) \to \mathcal{H}_2,
$$
defined as
\begin{align*}
T\varphi= 
\begin{cases}
0&\text{in $\Omega$},\\
\widehat{\varphi}-\varphi&\text{in $\Omega^c$},
\end{cases}                                
\end{align*}
where $\widehat{\varphi}$ is the harmonic extension into $\Omega^c$ of the restriction $\varphi |_{\partial\Omega}$.
Since $\mathcal H_1 = \ker T$, it is sufficient to show that $T$ is continuous.

 We have $T=T'-Id$, where 
\begin{align*}
T'\varphi= 
\begin{cases}
\varphi&\text{in $\Omega$},\\
\widehat{\varphi}&\text{in $\Omega^c$}
\end{cases}                                
\end{align*}
and $Id$ is the identity mapping. 
We recall the following estimate~ \cite[Proposition 1.7, p.360]{Taylor}
$$
 ||\widehat{\varphi}||_{H^1(\Omega^c,g)} \leqslant C  ||\varphi|_{\partial \Omega}||_{H^{1/2}(\partial \Omega,g)}.
$$
In the following, the letter $C$ denotes any constant depending only on $(M,g)$ and $\Omega$. Its exact value could differ from line to line. By the Trace Embedding Theorem one has
$$
|| \varphi_{|_{\partial \Omega}} ||_{H^{1/2}(\partial \Omega,g)} \leqslant C ||\varphi||_{H^{1}(\Omega,g)}.
$$
Finally, we have
$$
|| \varphi||_{H^1(\Omega,g)} \leqslant ||\varphi|| _{H^1(M,g)}. 
$$ 
All the above implies
$$
||\widehat{\varphi}||_{H^1(\Omega^c,g)} \leqslant C||\varphi||_{H^1(M,g)}.
$$
Therefore, one has
\begin{gather*}
||T'\varphi||^2_{H^1(M,g)}= ||T'\varphi||^2_{H^1(\Omega,g)}+ ||T'\varphi||^2_{H^1(\Omega^c,g)}= ||\varphi||^2_{H^1(\Omega,g)}+ ||\widehat{\varphi}||^2_{H^1(\Omega^c,g)} \leqslant \\ \leqslant ||\varphi||^2_{H^1(M,g)} +||\widehat{\varphi}||^2_{H^1(\Omega^c,g)} \leqslant ||\varphi||^2_{H^1(M,g)} +C^2||\varphi||^2_{H^1(M,g)} \leqslant C ||\varphi||^2_{H^1(M,g)},
\end{gather*}
which completes the proof that $T$ is continuous.

Finally, we prove that for any $\delta>0$  one has
 $$\int_M\langle\nabla \varphi, \nabla \psi \rangle_{\rho_\delta g} dv_{\rho_\delta g}=0, \forall \varphi \in \mathcal{H}_1, \psi \in \mathcal{H}_2.$$ Indeed,

\begin{gather*}
\int_M\langle\nabla \varphi, \nabla \psi \rangle_{\rho_\delta g} dv_{\rho_\delta g}= \int_\Omega \langle\nabla \varphi, \nabla \psi \rangle_{\rho_\delta g} dv_{\rho_\delta g}+ \int_{\Omega^c} \langle\nabla \varphi, \nabla \psi \rangle_{\rho_\delta g} dv_{\rho_\delta g}=\\= \int_\Omega \langle\nabla \varphi, \nabla \psi \rangle_{g} dv_{g}+ \delta^{\frac{n-2}{2}}\int_{\Omega^c} \langle\nabla \varphi, \nabla \psi \rangle_{g} dv_{g}=\\= \int_\Omega \Delta\varphi \psi~dv_g+ \delta^{\frac{n-2}{2}}\int_{\Omega^c} \Delta\varphi \psi~dv_g=0,
\end{gather*} 
since $\psi_{|_{\partial \Omega}}=0$. 

\end{proof}

For a function $\varphi \in H^1(M,g)$ we fix its decomposition $\varphi_1+\varphi_2$ with
\begin{align*}
\varphi_1=
 \begin{cases}
 \varphi&\text{in $\Omega$},\\
 \widehat{\varphi_{|_\Omega}}&\text{in $\Omega^c$}
 \end{cases}
\end{align*}
and $\varphi_2=\varphi_1-\varphi$.

 For the sake of simplicity we use the symbols $\lambda^\delta_k$ for $\lambda_k(\rho_\delta g)$, $g_\delta$ for $\rho_\delta g$ and $R_\delta$ for the Rayleigh quotient 
$$
R_\delta[\varphi]=\frac{\displaystyle\int_M|\nabla \varphi|^2_{g_\delta}dv_{g_\delta}}{\displaystyle\int_M \varphi^2 dv_{g_\delta}}.
$$

{\bf Claim 2.} \label{bound}
There exists a constant $C_k>0$ such that $\lambda^\delta_k \leqslant C_k$.

\begin{proof}
Theorem \ref{Kor} implies that there exists a constant $C(k)>0$ such that
$$
\Lambda_k(M, [g]) \leqslant C(k).
$$
By Lemma~\ref{identity} for every $\delta$ one has
$$
\lambda^\delta_k \Vol^{\frac{2}{n}}(M,g_\delta) \leqslant \Lambda_k(M, [g]) \leqslant C(k).
$$
Therefore,
$$
\lambda^\delta_k \leqslant \frac{C(k)}{\Vol^{\frac{2}{n}}(M, g_\delta)}=\frac{C(k)}{(\Vol(\Omega,g)+\delta^{\frac{n}{2}}\Vol(\Omega^c,g))^{\frac{2}{n}}} \leqslant \frac{C(k)}{(\Vol(\Omega,g))^{\frac{2}{n}}}=C_k
$$
\end{proof}

Let $W_k$ be the set of $k+1$-dimensional subspaces of $H^1(M,g_\delta)$ satisfying the condition that ${R_\delta}|_{W_k}\leqslant C_k$. We remark that according to Claim 2 the space spanned by the first $k+1$ eigenfunctions is in $W_k$, i.e. $W_k\ne\varnothing$.

{\bf Claim 3.} 
For every $\varphi \in V \in W_k$ there exists a constant $C>0$ such that
$$
\displaystyle\int_{\Omega^c}\varphi^2_2~dv_{g_\delta} \leqslant C\delta \displaystyle\int_M \varphi^2 dv_{g_\delta}. 
$$

\begin{proof}
By Claim 1 one has
 $$
 \int_M\langle\nabla \varphi_1, \nabla \varphi_2 \rangle_{g_\delta} dv_{g_\delta}=0.
 $$ 
Further, since $\varphi \in V \in W_k$ we have 
\begin{gather*}
C_k \geqslant R_\delta[\varphi]=\frac{\displaystyle\int_M|\nabla \varphi|^2_{g_\delta}dv_{g_\delta}}{\displaystyle\int_M \varphi^2 dv_{g_\delta}}=\frac{\displaystyle\int_M|\nabla \varphi_1|^2_{g_\delta}dv_{g_\delta}+\displaystyle\int_{M}|\nabla \varphi_2|^2_{g_\delta}dv_{g_\delta}}{\displaystyle\int_M \varphi^2 dv_{g_\delta}} \geqslant \\ \geqslant \frac{\displaystyle\int_{\Omega^c}|\nabla \varphi_2|^2_{g_\delta}dv_{g_\delta}}{\displaystyle\int_M \varphi^2 dv_{g_\delta}}=\frac{1}{\delta}\frac{\displaystyle\int_{\Omega^c}|\nabla \varphi_2|^2_{g}dv_{g}}{\displaystyle\int_M \varphi^2_2 dv_{g}}\frac{||\varphi_2||^2_{L^2(\Omega^c, g_\delta)}}{||\varphi||^2_{L^2(M, g_\delta)}}\geqslant \frac{\lambda^D_1(\Omega^c,g)}{\delta} \frac{||\varphi_2||^2_{L^2(\Omega^c, g_\delta)}}{||\varphi||^2_{L^2(M, g_\delta)}},
\end{gather*}
where $\lambda^D_1(\Omega^c,g)$ is the first non-zero Dirichlet eigenvalue of $(\Omega^c,g)$.
\end{proof}

{\bf Claim 4.} 
For every $\varphi \in V \in W_k$ and for every sufficiently small $\delta$ there exists a constant $C>0$ such that
$$
\int_{M}\varphi^2~dv_{g_\delta} \leqslant (1+C \sqrt{\delta}) \int_M \varphi^2_1 dv_{g_\delta}. 
$$

\begin{proof}
One has
\begin{align*}
||\varphi||^2_{L^2(M, g_\delta)}=\int_{\Omega^c}(\varphi_1+\varphi_2)^2dv_{g_\delta}+\int_{\Omega}\varphi^2_1dv_{g_\delta} \leqslant \Big(1+\frac{1}{\varepsilon}\Big) \int_{M}\varphi_2^2dv_{g_\delta}+(1+\varepsilon) \int_{M}\varphi^2_1dv_{g_\delta},
\end{align*}
for every $\varepsilon>0$. Applying Claim 3 we obtain
\begin{align*}
||\varphi||^2_{L^2(g_\delta)} \leqslant C\delta\Big(1+\frac{1}{\varepsilon}\Big) \int_{M}\varphi^2dv_{g_\delta}+(1+\varepsilon) \int_{M}\varphi^2_1dv_{g_\delta},
\end{align*}
and hence,
\begin{align*}
\Big(1-C\delta \Big(1+\frac{1}{\varepsilon}\Big)\Big)||\varphi||^2_{L^2(M, g_\delta)} \leqslant (1+\varepsilon) ||\varphi_1||^2_{L^2(M, g_\delta)}.
\end{align*}
Choosing $\varepsilon=\sqrt{\delta}$ completes the proof.  
\end{proof}

{\bf Claim 5.} \label{C3}
For every $\varphi \in V \in W_k$ and for every sufficiently small $\delta$ there exists a constant $C>0$ such that
$$
\int_{\Omega^c}\varphi^2_1~dv_{g} \leqslant C\int_{\Omega} \varphi^2_1 dv_{g}. 
$$

\begin{proof}
By the Sobolev Embedding Theorem one has
$$
||\varphi_1||_{L^2(\Omega^c,g)} \leqslant C ||\varphi_1||_{H^1(\Omega^c,g)}.
$$
Again by~\cite[Proposition 1.7, p.360]{Taylor}) one has
$$
 ||\varphi_1||_{H^1(\Omega^c,g)} \leqslant C  ||\varphi|_{\partial \Omega}||_{H^{1/2}(\partial \Omega,g)}.
$$
By the Trace Embedding Theorem one has
$$
|| \varphi_{|_{\partial \Omega}} ||_{H^{1/2}(\partial \Omega,g)} \leqslant C ||\varphi_1||_{H^{1}(\Omega,g)}.
$$
Altogether
\begin{align}\label{C'}
||\varphi_1||_{L^2(\Omega^c,g)} \leqslant C ||\varphi_1||_{H^{1}(\Omega,g)}.
\end{align}
Further, since $\varphi \in V \in W_k$ and $\int_M\langle\nabla \varphi_1, \nabla \varphi_2 \rangle_{g_\delta} dv_{g_\delta}=0$ one has
\begin{gather*}
C_k \geqslant R_\delta [\varphi]= \frac{\displaystyle\int_M|\nabla \varphi|^2_{g_\delta}dv_{g_\delta}}{\displaystyle\int_M \varphi^2 dv_{g_\delta}}=\frac{\displaystyle\int_M|\nabla \varphi_1|^2_{g_\delta}dv_{g_\delta}+\displaystyle\int_{M}|\nabla\varphi_2|^2_{g_\delta}dv_{g_\delta}}{\displaystyle\int_M \varphi^2 dv_{g_\delta}} \geqslant \frac{\displaystyle\int_{\Omega}|\nabla \varphi_1|^2_{g}dv_{g}}{\displaystyle\int_M \varphi^2 dv_{g_\delta}},
\end{gather*}
hence,
\begin{gather*}
\int_\Omega |\nabla\varphi_1|_{g}^2~dv_{g} \leqslant C_k \int_M \varphi^2 dv_{g_\delta},
\end{gather*}
and by Claim 4 one gets
\begin{gather*}
\int_\Omega |\nabla\varphi_1|_{g}^2~dv_{g} \leqslant C_k (1+C\sqrt{\delta}) ||\varphi_1||^2_{L^2(M, g_\delta)}=\\=C_k (1+C\sqrt{\delta}) (||\varphi_1||^2_{L^2(\Omega, g)}+\delta^{n/2} ||\varphi_1||^2_{L^2(\Omega^c, g)}). 
\end{gather*}
Plugging the latter in \eqref{C'} we obtain
\begin{gather*}
 ||\varphi_1||^2_{L^2(\Omega^c,g)}\leqslant C ||\varphi_1||^2_{H^{1}(\Omega,g)}=C( ||\varphi_1||^2_{L^{2}(\Omega,g)}+||\nabla \varphi_1||^2_{L^{2}(\Omega,g)}) \leqslant \\ \leqslant C\big(||\varphi_1||^2_{L^{2}(\Omega,g)}+C_k (1+C\sqrt{\delta}) (||\varphi_1||^2_{L^2(\Omega, g)}+ \delta^{n/2}||\varphi_1||^2_{L^2(\Omega^c, g)})\big).
\end{gather*}
Rearranging the terms yields the required inequality.
\end{proof}
By Claim 4 for every $\varphi \in V \in W_k$ and  $\int_M\langle\nabla \varphi_1, \nabla \varphi_2 \rangle_{g_\delta} dv_{g_\delta}=0$ one has
\begin{gather*}
R_\delta[\varphi]=\frac{\displaystyle\int_M|\nabla \varphi|^2_{g_\delta}dv_{g_\delta}}{\displaystyle\int_M \varphi^2 dv_{g_\delta}}=\frac{\displaystyle\int_M|\nabla \varphi_1|^2_{g_\delta}dv_{g_\delta}+\displaystyle\int_{M}|\nabla \varphi_2|^2_{g_\delta}dv_{g_\delta}}{\displaystyle\int_M \varphi^2 dv_{g_\delta}} \geqslant \\ \geqslant \frac{1}{1+C\sqrt{\delta}} \frac{\displaystyle\int_M|\nabla \varphi_1|^2_{g_\delta}dv_{g_\delta}+\int_{M}|\nabla \varphi_2|^2_{g_\delta}dv_{g_\delta}}{\displaystyle\int_M \varphi_1^2 dv_{g_\delta}} \geqslant \\ \geqslant \frac{1}{1+C\sqrt{\delta}} \frac{\displaystyle\int_M|\nabla \varphi_1|^2_{g_\delta}dv_{g_\delta}}{\displaystyle\int_M \varphi_1^2 dv_{g_\delta}}=\frac{1}{1+C\sqrt{\delta}} \frac{\displaystyle\int_M|\nabla \varphi_1|^2_{g_\delta}dv_{g_\delta}}{\displaystyle\int_{\Omega} \varphi_1^2 dv_{g}+\delta^{\frac{n}{2}}\int_{\Omega^c}\varphi_1^2 dv_{g}}.
\end{gather*}
By Claim 5 we then have
\begin{gather*}
R_\delta[\varphi] \geqslant \frac{1}{(1+\delta^{\frac{n}{2}}C)(1+C\sqrt{\delta})} \frac{\displaystyle\int_M|\nabla \varphi_1|^2_{g_\delta}dv_{g_\delta}}{\displaystyle\int_{\Omega} \varphi_1^2 dv_{g}} \geqslant \\ \geqslant \frac{1}{(1+\delta^{\frac{n}{2}}C)(1+C\sqrt{\delta})} \frac{\displaystyle\int_\Omega|\nabla \varphi_1|^2_{g}dv_{g}}{\displaystyle\int_{\Omega} \varphi_1^2 dv_{g}}=\frac{1}{(1+\delta^{\frac{n}{2}}C)(1+C\sqrt{\delta})} \frac{\displaystyle\int_\Omega|\nabla \varphi|^2_{g}dv_{g}}{\displaystyle\int_{\Omega} \varphi^2 dv_{g}} \geqslant \\ \geqslant \frac{1}{(1+\delta^{\frac{n}{2}}C)(1+C\sqrt{\delta})} R^N_{(\Omega,g)}[\varphi|_{\Omega}],
\end{gather*}
where $R^N_{(\Omega,g)}$ denotes the Rayleigh quotient for the Neumann problem in the domain $(\Omega,g)$.
Let $V=\myspan\langle \psi_0,\ldots, \psi_k \rangle$, where $\psi_i$ is in the $i$-th eigenspace of $(M,g_\delta)$. Then
\begin{equation}\label{before}
\begin{split}
\lambda^\delta_k=\max_{\varphi \in V}R_\delta[\varphi] &\geqslant \frac{1}{(1+ \delta^{\frac{n}{2}} C)(1+C\sqrt{\delta})} \max_{\varphi \in V}R^N_{(\Omega,g)}[\varphi|_{\Omega}] \geqslant \\ &\geqslant \frac{1}{(1+ \delta^{\frac{n}{2}} C)(1+C\sqrt{\delta})}\lambda^N_k(\Omega,g),
\end{split}
\end{equation}
since by unique continuation the restriction to $\Omega$ of the functions $\psi_i$ form the space of the same dimension. 
Taking the $\liminf$ as $\delta \to 0$ in~\eqref{before} competes the proof.
\end{proof}

Using Proposition~\ref{subdomain} one gets the following corollary. 
\begin{corollary}
\label{Neumann cor2}
Let $(M, g)$ be a closed compact Riemannian manifold. Consider a sequence $\{ K_n \}$ of smooth domains $K_n \subset M$ such that
\begin{itemize}
\item $K_r \subset K_s$ $\forall r>s$;
\item $\cap_n K_n=\{p_1,\ldots,p_l\}$ for some points $p_1,\ldots,p_l \in M$.
\end{itemize} 
Then one has
$$
\lim_{n \to \infty}\Lambda^N_k(M \setminus K_n, [g])= \Lambda_k(M, [g]).
$$ 
\end{corollary}
\begin{proof}
Proposition~\ref{subdomain} implies that
$$
\limsup_{n \to \infty}\Lambda^N_k(M \setminus K_n, [g]) \leqslant \Lambda_k(M, [g]).
$$
It remains to show that 
$$
\Lambda_k(M, [g]) \leqslant \liminf_{n \to \infty}\Lambda^N_k(M \setminus K_n, [g]).
$$
Let $g^m$ be a maximizing sequence for the functional $\Lambda_k(M, [g])$. Then for a fixed $m$ we consider geodesic balls $B_{\epsilon_n}(p_i)$ of radius $\epsilon_n\to 0$ in metric $g^m$ centred at the points $p_1,\ldots,p_l \in M$ such that $K_n \subset \cup^l_{i=1}B_{\epsilon_n}(p_i)$. Then $M\setminus \cup^l_{i=1}B_{\epsilon_n}(p_i) \subset M\setminus K_n$ and Proposition~\ref{subdomain} implies that
\begin{gather}\label{420}
\Lambda^N_k(M\setminus K_n,[g]) \geqslant \Lambda^N_k(M\setminus  \cup^l_{i=1}B_{\epsilon_n}(p_i),[g])\geqslant \bar\lambda^N_k(M\setminus \cup^l_{i=1}B_{\epsilon_n}(p_i),g^m).
\end{gather}
Note that $\Vol(M \setminus \cup^l_{i=1}B_{\epsilon_n}(p_i),g^m)\to\Vol(M,g^m)$ as $n\to\infty$ and by Lemma~\ref{Neumann conv} one has $\lambda^N_k(M \setminus \cup^l_{i=1}B_{\epsilon_n}(p_i),g^m)\to\lambda_k(M,g^m)$. Hence, $\bar\lambda^N_k(M\setminus  \cup^l_{i=1}B_{\epsilon_n}(p_i),g^m) \to \bar\lambda_k(M,g^m)$ as $n\to\infty$. 
Taking $\liminf_{n\to\infty}$ in~\eqref{420} one then gets
$$
\liminf_{n\to\infty}\Lambda^N_k(M\setminus K_n,[g]) \geqslant \bar\lambda_k(M,g^m).
$$
Passing to the limit as $m\to\infty$ completes the proof.
\end{proof}

\subsection{Disconnected manifolds.}
\begin{lemma}
\label{disconnected}
Let $(\Omega,g) = \sqcup_{i=1}^s(\Omega_i,g_i)$ be a disjoint union of Riemannian manifolds of dimension $n$ with smooth boundary. Then for all $k>0$ one has
$$
\Lambda^N_k(\Omega,[g])^{\frac{n}{2}} = \max_{\sum\limits_{i=1}^s k_i=k,\,\,\,k_i>0}\,\,\sum_{i=1}^s\Lambda^N_{k_i}(\Omega_i,[g_i])^{\frac{n}{2}}.
$$
Similarly, if $(M,g)=\sqcup_{i=1}^s(M_i,g_i)$ is a disjoint union of closed Riemannian manifolds of dimension $n$, then one has
$$
\Lambda_k(M,[g])^{\frac{n}{2}} = \max_{\sum\limits_{i=1}^s k_i=k,\,\,\,k_i>0}\,\,\sum_{i=1}^s\Lambda_{k_i}(M_i,[g_i])^{\frac{n}{2}}.
$$
\end{lemma}
\begin{proof}
The proof is reminiscent of the argument due to Wolf and Keller~\cite{WK94}. The differences between the proofs of two equalities are cosmetic, we only present the proof of the first equality.

{\bf Inequality $\geqslant$.} 

Fix the indices $k_i>0$ satisfying $\sum k_i=k$.  
Let $\{g_i^m\}$ be a maximizing sequence of metrics such that  $\bar\lambda^N_{k_i}(\Omega_i,g^m_i)\to\Lambda^N_{k_i}(\Omega_i,[g_i])$. Up to a rescaling one can assume that $\lambda^N_{k_i}(\Omega_i, g^m_i)=\Lambda^N_k(\Omega, [g])$. Then, one has
$$
\Vol(\Omega_i,g^m_i)\to \frac{\Lambda^N_{k_i}(\Omega_i, [g_i])^\frac{n}{2}}{\Lambda^N_{k}(\Omega,[g])^\frac{n}{2}}
$$

Consider a sequence of metrics $\{g^m\}$ on $\Omega$ defined as $g^m|_{\Omega_i}=g^m_i$. Since the spectrum of disjoint union is the union of spectra of each component, then for large enough $m$ one has that $\lambda^N_k(\Omega,g^m) = \Lambda^N_k(\Omega,[g])$. At the same time, by definition of $\Lambda^N_k(\Omega,[g])$ one has
$$
\Lambda^N_k(\Omega,[g])\Vol(\Omega, g^m)^\frac{2}{n}=\lambda^N_k(\Omega,g^m)\Vol(\Omega, g^m)^\frac{2}{n}\leqslant \Lambda^N_{k}(\Omega,[g]),
$$
i.e. $\Vol(\Omega,g^m) \leqslant 1$. Therefore, one obtains
$$
1\geqslant \Vol(\Omega,g^m) = \sum_i\Vol(\Omega_i,g^m_i)\to \frac{\sum_i \Lambda^N_{k_i}(\Omega_i, [g_i])^\frac{n}{2}}{\Lambda^N_{k}(\Omega,[g])^\frac{n}{2}}.
$$
Passing to the limit $m\to\infty$ yields the inequality.

{\bf Inequality $\leqslant$.} 

Assume the contrary, i.e. 
\begin{equation}
\label{assumption}
\Lambda^N_k(\Omega,[g])^{\frac{n}{2}} > \max_{\sum\limits_{i=1}^s k_i=k,\,\,\,k_i>0}\,\,\sum_{i=1}^s\Lambda^N_{k_i}(\Omega_i,[g_i])^{\frac{n}{2}}.
\end{equation}
Let $\{g^m\}$ be a maximizing sequence of metrics of volume $1$ such that $\lambda^N_{k}(\Omega,g^m)\to\Lambda^N_{k}(\Omega,[g])$. Let $g_i^m$ be a restriction of $g^m$ to $\Omega_i$. Further, let $d^m_i$ be the largest number such that $\lambda^N_{d^m_i}(\Omega_i,g_i^m)<\Lambda^N_k(\Omega,[g])$ and $\limsup_{m\to\infty}\lambda^N_{d^m_i}(\Omega_i,g_i^m)<\Lambda^N_k(\Omega,[g])$
and $V^m_i$ be $\Vol(\Omega_i,g^m_i)$.
Then one has  $d^m_i\leqslant k$
and $V^m_i\leqslant 1$.
Therefore, up to a choice of a subsequence one can assume that $d^m_i = d_i$ does not depend on $m$
and $V^m_i\to V_i$ as $m\to\infty$.

We claim that $\sum_i(d_i+1)\geqslant k+1$. Otherwise, by~\eqref{assumption} and definition of $d_i$ one has
$$
\Lambda^N_k(\Omega,[g])^\frac{n}{2}\sum_i V_i\leqslant\sum_i\limsup_{m\to\infty}\bar\lambda^N_{d_i+1}(\Omega_i,g^m_i)^\frac{n}{2}\leqslant\sum_i\Lambda^N_{d_i+1}(\Omega_i,[g])^{\frac{n}{2}}<\Lambda^N_k(\Omega,[g])^{\frac{n}{2}}.
$$
Since $g^m$ are of unit volume, one has $\sum_i V_i=1$. 
Thus, one arrives at $\Lambda^N_k(\Omega,[g])^\frac{n}{2}<\Lambda^N_k(\Omega,[g])^\frac{n}{2},$ which is a contradiction.

Therefore, one has $\sum(d_i+1)\geqslant k+1$. Since the spectrum of a union is a union of spectra, one has $\lambda^N_k(\Omega,g^m)\in\bigcup_i\{\lambda_0(\Omega_i,g^m_i),\ldots,\lambda_{d_i}(\Omega_i,g^m_i)\}$, i.e. 
$$
\Lambda^N_k(\Omega,g)=\limsup_{m\to\infty}\lambda^N_k(\Omega,g^m)\leqslant\max_i\limsup_{m\to\infty}\lambda_{d_i}(\Omega_i,g^m_i)<\Lambda^N_k(\Omega,[g]).
$$
Since $g^m$ are of unit volume we arrive at a contradiction.
\end{proof}

Finally, as a corollary of Lemma~\ref{identity2}, Proposition~\ref{subdomain} and Lemma~\ref{disconnected} one obtains.

\begin{lemma}\label{omega_i}
Let $(M,g)$ be a closed Riemannian manifold of dimension $n$. Consider a set of pairwise disjoint smooth domains $\{\Omega_i\}^s_{i=1}$ in $M$ such that $M=\bigcup^s_{i=1} \overline\Omega_i$. Then one has
$$
\Lambda_k(M, [g])^{\frac{n}{2}} \geqslant \max_{\sum_{i=1}^s k_i=k,\,\,\,k_i\geqslant0} \sum^s_{i=1} \Lambda^N_{k_i}(\Omega_i, [g])^{\frac{n}{2}}.
$$

If $M$ is compact with non-empty boundary, then one has
$$
\Lambda^N_k(M, [g])^{\frac{n}{2}} \geqslant \max_{\sum_{i=1}^s k_i=k,\,\,\,k_i\geqslant 0} \sum^s_{i=1} \Lambda^N_{k_i}(\Omega_i, [g])^{\frac{n}{2}}.
$$
\end{lemma}
\begin{proof}
Once again, we only give a proof for the closed case.

Fix indices $k_i\geqslant 0$ such that $\sum_{i=1}^s k_i=k$. 
Let $I = \{i\,|\,k_i>0\}$ and set $\Omega_1 = \cup_{i\in I}\overline\Omega_i\subset M$, $(\Omega_2,h) = \sqcup_{i\in I}(\overline\Omega_i,g_{\overline\Omega_i})$. Applying in order: Proposition~\ref{subdomain}, Lemma~\ref{identity2} and Lemma~\ref{disconnected}, one obtains
$$
\Lambda_k(M,[g])\geqslant \Lambda^N_k(\Omega_1,[g])\geqslant \Lambda_k(\Omega_2,[h])\geqslant \sum_{i\in I} \Lambda^N_{k_i}(\Omega_i, [g])^{\frac{n}{2}}=\sum^s_{i=1} \Lambda^N_{k_i}(\Omega_i, [g])^{\frac{n}{2}},
$$
where in the last equality we used that $\Lambda_0(\Omega_j,[g])=0$ for any $j$.
\end{proof}

\section{Proof of Theorem~\ref{conf&conv}} 
\label{appendix}

We remind the reader that as $n\to\infty$ one has $s$ $1$-sided geodesics and $\widetilde s$ 2-sided geodesics collapse and the canonical representative metric $h_n \in c_n$ is hyperbolic if $\chi(\Sigma_\gamma)<0$ and is flat if $\chi(\Sigma_\gamma)=0$. We start with the hyperbolic case and discuss the flat case at the end of the section.

We introduce the following notations
\begin{itemize}
\item $\mathcal{C}^n_i$ for collars of $1$-sided collapsing geodesics, $i=1,\ldots,s$. Their width is denoted by $w^n_i$ 
\item $\widetilde{\mathcal{C}}^n_i$ for collars of $2$-sided collapsing geodesics, $i=1,\ldots,\widetilde s$. Their width is denoted by $\widetilde w^n_i$
\item $M^n_j$ for a connected component of $M \setminus (\cup^s_{i=1} \mathcal{C}^n_i\bigcup \cup_{i=1}^{\widetilde s}\widetilde{\mathcal C}^n_i)$
\item for $-\widetilde w^n_i\leqslant a\leqslant b\leqslant \widetilde w^n_i$, we denote $\widetilde{\mathcal{C}}^n_i(a,b)\subset \widetilde{\mathcal{C}}^n_i$ the subset $\{(t, \theta)\,|a\leqslant t\leqslant b\}$
\item for $0\leqslant a\leqslant b\leqslant w^n_i$, we denote $\mathcal{C}^n_i(a,b)\subset\mathcal{C}^n_i$ the subset 
$$
\{(t, \theta)\,|a\leqslant t\leqslant b\}\cup\{(t, \theta)\,|-b\leqslant t\leqslant -a\}/\sim.
$$ 
It is a M\"obius band if $a=0$ and cylinder otherwise.


\item Let $\alpha^n = \cup_{i=1}^s\alpha^n_i\bigcup\cup_{i=1}^{\widetilde s}\{\alpha^n_{j,-},\alpha^n_{j,+}\}$, where $0\leqslant\alpha^n_i\leqslant w^n_i$ and $-\widetilde w^n_i\leqslant \alpha^n_{i,-}\leqslant\alpha^n_{i,+} \leqslant \widetilde w^n_i$.
We denote by $M^n_j(\alpha^n)$ the connected component of
 $$
 M \setminus \Big(\cup^{s}_{i=1} \mathcal{C}_i^n(0,\alpha^n_{i}) \bigcup \cup^{\tilde s}_{i=1} \widetilde{\mathcal{C}}_i^n(\alpha^n_{i,-},\alpha^n_{i,+}) \Big)
 $$
  which contains $M_j^n$;

\item $a_n \ll b_n$ for two sequences $\{a_n\}$ and $\{b_n\}$ satisfying $a_n, b_n \to +\infty$  and $\frac{a_n}{b_n} \to 0$ as $n \to \infty$. 
\end{itemize}

\subsection{Inequality $\geqslant$.} We start with proving the inequality

\begin{equation}\label{aim}
\begin{split}
&\liminf_{n \to \infty} \Lambda_k (M, c_n) \geqslant\\ &\max \Big(\sum^{\widetilde m}_{i=1} \Lambda_{\widetilde k_i}(\widetilde\Sigma_{\widetilde\gamma_{i}}, c_\infty)+\sum^{m}_{i=1} \Lambda_{k_{i}}(\Sigma_{\gamma_{i}},c_\infty) + \sum_{i=1}^{\widetilde s}\Lambda_{\widetilde r_i}(\mathbb{S}^2) + \sum_{i=1}^s\Lambda_{r_i}(\mathbb{RP}^2) \Big),
\end{split}
\end{equation}

Consider the domains $\mathcal{C}^n_i(0,\alpha^n_i)$ for $1 \leqslant i \leqslant s$, $\widetilde{\mathcal{C}}^n_i(\alpha_{i,-}^n,\alpha_{i,+}^n)$ for $1 \leqslant i \leqslant \tilde s$, where $w^n_i-\alpha_i^n \ll w^n_i,$ $\alpha_i^n\to\infty$ and $ \widetilde w^n_i - \alpha_{i,\pm}^n \ll \widetilde w^n_i,$ $\alpha_{i,\pm}^n\to\infty$ and the domain $M_j^n(\alpha^n)$. By Lemma~\ref{omega_i} we have
\begin{equation}
\label{aim2}
\begin{split}
&\Lambda_k(M, c_n) \geqslant \max \Big(\sum^{s}_{i=1}\Lambda^N_{r_{i}}(\mathcal{C}^n_i(0, \alpha_{i}^n), c_n)+\\&\sum^{\tilde s}_{i=1}\Lambda^N_{\tilde r_{i}}(\widetilde{\mathcal{C}}^n_i(\alpha_{i,-}^n,\alpha_{i,+}^n), c_n)+\sum^{m+\widetilde m}_{j=1}\Lambda^N_{k_{j}}(M_j^n(\alpha^n), c_n)\Big).
\end{split}
\end{equation}

 For $1 \leqslant i \leqslant \tilde s$ we define the conformal maps $\widetilde\Psi_i^n\colon (\widetilde{\mathcal{C}}^n_i(\alpha_{i,-}^n,\alpha_{i,+}^n), c_n) \to (\mathbb{S}^2, [g_{can}])$ as 
$$
\widetilde\Psi_i^n(t,\theta)=\frac{1}{e^{2t}+1}(2e^{t}\cos\theta, 2e^{t}\sin\theta,e^{2t}-1).
$$
For $1 \leqslant i \leqslant s$ we define the conformal maps $\Psi_i^n\colon(\mathcal{C}^n_i(0, \alpha_{i}^n), c_n) \to (\mathbb{RP}^2, [g_{can}])$ as the maps, such that their lift to orientable double covers is given by the same formula as $\widetilde\Psi_i^n$. Finally, we take a restriction of a diffeomorphism $\Psi_n^{-1}$ given by Proposition~\ref{D-M} to obtain  a conformal map $\check\Psi_j^n\colon({M}_j^n(\alpha^n),c_n)\to (M_\infty,\Psi_n^*c_n)$.

Let $\Omega_{i}^n \subset \mathbb{RP}^2$, $\widetilde\Omega_{i}^n \subset \mathbb{S}^2$ and $\check\Omega_{j}^n \subset M_\infty$ be the the images of $\Psi_i^n$, $\widetilde\Psi_i^n$ and $\check\Psi_j^n$ respectively.
Since $\alpha^n_i,\,\alpha^n_{i,\pm}\to\infty$, the domains $\Omega_{i}^n$ and $\widetilde\Omega_{i}^n$ exhaust $\mathbb{RP}^2$ and $\mathbb{S}^2$ respectively. The corresponding statement for $\check\Omega_{j}^n$ is the content of the following lemma.

\begin{lemma} 
Let $M^\infty_j$ be the connected component $\check\Psi^n_j(M^n_j)\subset M_\infty$. Then the domains $\check\Omega_j^n$ exhaust $M^\infty_j$. 
\end{lemma}
\begin{proof}
Let $M_\infty=M^\infty_{\geqslant \varepsilon}\cup M^\infty_{<\varepsilon}$ be an $\varepsilon$-thick-thin decomposition  of $(M_\infty, h_\infty)$. For a sufficiently small $0<\varepsilon < \arcsinh(1)$ the $\varepsilon$-thin part  $M^\infty_{<\varepsilon}$ is nothing but subcollars of cusps (see~\cite[Proposition IV.4.2]{MR1451624}). For the surface $(M,h_n)$ we set $l^n_i$ for the length of the $i$-th 1-sided pinching geodesic and $\tilde l^n_{j}$ for the length of the $j$-th 2-sided pinching geodesic, where as before $i=1,\ldots,s$ and $j=1,\ldots,\tilde s$. Consider the diffeomorphism $(\Psi^n)^{-1}\colon M \to M_\infty$. From \cite[formula (4.12)]{Zhu} it follows that for a fixed $\varepsilon$ and for all $\varepsilon_1<\varepsilon$ there exists a number $N_1$ such that for all $n>N_1$ one has 
\begin{align}\label{1}
\cup^{\tilde s}_{j=1}\widetilde{\mathcal C}^n_j(\tilde \beta^n_{j}(\varepsilon_1),\tilde \beta^n_{j}(\varepsilon_1)) \bigcup \cup^{s}_{i=1}\mathcal C^n_i(0,\beta^n_i(\varepsilon_1)) \subseteq \Psi^n\left(M^\infty_{<\varepsilon}\right),
\end{align}
where 
$$
\tilde \beta^n_{j}(\varepsilon_1)=\frac{\pi}{\tilde l^n_{j}}\left(\pi-2\arcsin\left(\frac{\sinh\left(\tilde l^n_{j}/2\right)}{\sinh \varepsilon_1}\right)\right)
$$
and 
$$
\beta^n_i(\varepsilon_1)=\frac{\pi}{2l^n_i}\left(\pi-2\arcsin\left(\frac{\sinh l^n_i}{\sinh \varepsilon_1}\right)\right).
$$
Since $\frac{w^n_i}{l^n_i}\to 1$ and $\frac{\widetilde w^n_i}{\tilde l^n_i}\to 1$, there exists a number $N_2$ such that for every $n>N_2$ one has $\alpha^n_{j,\pm}<\tilde \beta^n_{j}(\varepsilon_1)$ and $\alpha^n_i< \beta^n_i(\varepsilon_1)$. Therefore, for all $n>N_2$ and for all $i,\,j$ one obtains
\begin{align}\label{2}
%
\widetilde{\mathcal C}^n_j(\alpha^n_{j,-},\alpha^n_{j,+})\subset\widetilde{\mathcal C}^n_i(\tilde \beta^n_{j}(\varepsilon_1),\tilde \beta^n_{j}(\varepsilon_1)),\quad \mathcal C^n_i(0,\alpha^n_i)\subset \mathcal C^n_i(0,\beta^n_i(\varepsilon_1))
\end{align}
Then (\ref{1}) and (\ref{2}) imply that for all $n>\max\{N_1,N_2\}$ one has
$$
M\setminus (\Psi^n)^{-1}M^\infty_{<\varepsilon} \subseteq M \setminus \Big(\cup^{\tilde s}_{i=1}\widetilde{\mathcal C}^n_i(\alpha^n_{i,-},\alpha^n_{i,+}) \bigcup \cup^{s}_{i=1}\mathcal C^n_i(0,\alpha^n_i)\Big)=\cup^r_{j=1}M^n_j(\alpha^n).
$$
Applying $\Psi^n$ we then get
$$
M^\infty_{\geqslant\varepsilon} \subseteq \cup^r_{j=1}\check\Omega^n_j.
$$
Since the domain $M^\infty_{\geqslant\varepsilon}$ exhausts $M_\infty$ as $\varepsilon$ goes to 0 we get the same for the domains $\check\Omega^n_j$ as $n$ goes to $\infty$ and the claim follows.\end{proof}
Applying the conformal transformations to~\eqref{aim2} one has
\begin{equation}
\label{aim3}
\begin{split}
&\Lambda_k(M, c_n) \geqslant\\ & \max\left(\sum^s_{i=1}\Lambda^N_{r_{i}}(\Omega_{i}^n, [g_{can}])+\sum^{\tilde s}_{i=1}\Lambda^N_{\tilde r_{i}}(\widetilde \Omega_{i}^n, [g_{can}])+\sum^{m+\widetilde m}_{j=1}\Lambda^N_{k_{j}}(\check\Omega_j^n, [(\Psi^n)^*h_n])\right).
\end{split}
\end{equation}
By Corollary~\ref{Neumann cor2} one has that the first two terms on the right hand side converge to $\Lambda_{r_i}(\mathbb{RP}^n)$ and $\Lambda_{\widetilde r_i}(\mathbb{S}^n)$ respectively.
\begin{lemma}
Let $\widehat{M_j^\infty}\subset \widehat{M_\infty}$ be a closure of $M_j^\infty$. Then for all $m$ one has 
$$
\liminf_{n\to\infty}\Lambda^N_{m}(\check\Omega_j^n, [(\Psi^n)^*h_n])\geqslant \Lambda_m(\widehat{M_j^\infty}, [\widehat{h_\infty}]).
$$
\end{lemma}
\begin{proof}

Fix $\varepsilon>0$. An application of Corollary~\ref{Neumann cor2} to a compact exhaustion of $M^\infty_j$ yields the existence of a compact $K\subset M^\infty_j\subset \widehat{M_j^\infty}$ such that 
$$
|\Lambda_m(\widehat{M_j^\infty},[\widehat{h_\infty}]) - \Lambda_m(K,[\widehat{h_\infty}])|<\varepsilon.
$$  
Since $\check\Omega_j^n$ exhaust $M^\infty_j$, then for all large enough $n$ one has $K\subset\check\Omega_j^n$.
Then, by Proposition~\ref{subdomain}
$$
\Lambda^N_{m}(\check\Omega_j^n, [(\Psi^n)^*h_n])\geqslant \Lambda^N_{m}(K, [(\Psi^n)^*h_n]).
$$
Taking $\liminf$ of both sides in the above inequality and using Proposition~\ref{N-cont} yields
$$
\liminf_{n\to\infty}\Lambda^N_{m}(\check\Omega_j^n, [(\Psi^n)^*h_n])\geqslant \Lambda^N_{m}(K, [\widehat{h_\infty}])> \Lambda_m(\widehat{M_j^\infty},[\widehat{h_\infty}])-\varepsilon.
$$
Since $\varepsilon$ is arbitrary, this completes the proof.
\end{proof}

Finally, taking $\liminf_{n\to\infty}$ in~\eqref{aim3} completes the proof of~\eqref{aim}.


\subsection{Inequality $\leqslant$.} We proceed with the inverse inequality,
\begin{equation}\label{aim'}
\begin{split}
&\limsup_{n \to \infty} \Lambda_k (M, c_n) \leqslant\\ &\max \Big(\sum^{\widetilde m}_{i=1} \Lambda_{\widetilde k_i}(\widetilde\Sigma_{\widetilde\gamma_{i}}, c_\infty)+\sum^{m}_{i=1} \Lambda_{k_{i}}(\Sigma_{\gamma_{i}},c_\infty) + \sum_{i=1}^{\widetilde s}\Lambda_{\widetilde r_i}(\mathbb{S}^2) + \sum_{i=1}^s\Lambda_{r_i}(\mathbb{RP}^2) \Big),
\end{split}
\end{equation}
In orientable case, this is essentially proved in~\cite[Section 7]{petrides2017existence}. Below we outline the ideas of the proof and show the necessary modifications in the non-orientable case.

Let us choose a subsequence $c_{n_m}$ such that $$\lim_ {n_m \to \infty} \Lambda_k (M, c_{n_m})=\limsup_{n \to \infty} \Lambda_k (M, c_n).$$ We immediately relabel the subsequence and denote it by $\{c_n\}$. This way we can choose further subsequences without changing the value of $\limsup$.

{\bf Case 1.} Suppose that up to a choice of a subsequence the following inequality holds
\begin{align*}
\Lambda_k(M, c_n) > \Lambda_{k-1}(M, c_n) +8\pi.
\end{align*}
Then by \cite[Theorem 2]{petrides2017existence} in the conformal class $c_n$ there exists a unit volume metric $g_n$ induced from a harmonic immersion $\Phi_n$ to some $N(n)$-dimensional sphere $\mathbb{S}^{N(n)}$, i.e.
$$
g_n=\frac{|\nabla \Phi_n|^2_{h_n}}{\Lambda_k(M, c_n)}h_n
$$
and such that $\lambda_k(g_n)=\Lambda_k(M, c_n)$. Here the metric $h_n$ is the canonical representative in the conformal class $c_n$. It is known that for any compact surface the multiplicity of $\lambda_k(g_n)$ is bounded from above by a constant depending only on $k$ and $\gamma$ (see for instance~\cite{cheng1976eigenfunctions} for orientable surfaces and \cite{besson1980multiplicite, Nad98} for non-orientable surfaces). Therefore, one can choose the number $N(n)$ large enough such that $N(n)$ does not depend on $n$. 

Assume that for the sequence $\{c_n\}$ the following inequality holds

\begin{equation}\label{gap}
\begin{split}
&\lim_{n \to \infty} \Lambda_k (M, c_n) >\\ &\max \left(\sum^{\widetilde m}_{i=1} \Lambda_{\widetilde k_i}(\widetilde\Sigma_{\widetilde\gamma_{i}}, c_\infty)+\sum^{m}_{i=1} \Lambda_{k_{i}}(\Sigma_{\gamma_{i}},c_\infty) + \sum_{i=1}^{\widetilde s}\Lambda_{\widetilde r_i}(\mathbb{S}^2) + \sum_{i=1}^s\Lambda_{r_i}(\mathbb{RP}^2) \right).
\end{split}
\end{equation}


\begin{proposition} \label{sequences}
 For $1 \leqslant i \leqslant s$ there exist integers $t_{i} \geqslant 0$, non-negative sequences $\{a_{i,l}^n\}, \{b_{i,l}^n\}$ with $1 \leqslant l \leqslant t_{i}$ and a sequence $\{\alpha^n_i\}$ such that
\begin{gather*}
0\leqslant a_{i,t_i}^n \ll b_{i,t_i}^n \ll ... \ll a_{i,1}^n \ll b_{i,1}^n \ll a_{i,0}^n=\alpha_{i}^n \ll w_i^n
\end{gather*}
and 
$$
m_{i,l}=\lim_{n \to \infty} \Vol(\mathcal{C}_i^n(a_{i,l}^n,b_{i,l}^n), g_n)>0.
$$
Similarly for $1 \leqslant i \leqslant \tilde s$ there exist integers $\tilde t_{i} \geqslant 0$, sequences $\{a_{i,l}^n\}, \{b_{i,l}^n\}$ where $1 \leqslant l \leqslant \tilde t_{i}$ and sequences $\{\alpha_{i,\pm}^n\}$ such that 
\begin{gather*}
-w_i^n \ll \alpha_{i,-}^n=b_{i,0}^n \ll a_{i,1}^n \ll b_{i,1}^n \ll \ldots \ll a_{i,t_i}^n \ll b_{i,t_{i}+1}^n \ll a_{i,t_{i+1}}^n=\alpha_{i,+}^n \ll w_i^n
\end{gather*}
and 
$$
\widetilde m_{i,l}=\lim_{n \to \infty} \Vol(\widetilde{\mathcal{C}}_i^n(a_{i,l}^n,b_{i,l}^n), g_n)>0.
$$  
Moreover, there exists a set $J \subset \{1,\ldots,m+\widetilde m\}$ such that for every $j \in J$ one has
$$
m_{j}=\lim_{n \to \infty} \Vol(M_j^n(\alpha_n), g_n)>0
$$
satisfying 
$$
\sum^s_{i=1} \sum^{t_i}_{l=1}m_{i,l}+\sum^{\tilde s}_{i=1} \sum^{\tilde t_i}_{l=1}\widetilde m_{i,l}+\sum_{j\in J}m_j=1,
$$
with $\sum^s_{i=1}t_i+\sum^{\tilde s}_{i=1}{\tilde t_i}\leqslant k$ is maximal.
\end{proposition}
\begin{proof}
The proof follows the proofs of Claim 16, Claim 17 by \cite{petrides2017existence}. Precisely, denying the proposition one can construct $k+1$ test-functions such that $\lambda_k(g_n) \leqslant o(1)$ which contradicts inequality~\eqref{bruno}.
\end{proof}

We proceed with considering a sequence $\{d_{i,l}^n\}$ where $1\leqslant i \leqslant s$ and $1 \leqslant l \leqslant t_i$ such that 
$$
\lim_{n \to \infty} \Vol(\mathcal{C}_i^n(a_{i,l}^n,d_{i,l}^n),g_n)= \lim_{n \to \infty} \Vol(\mathcal{C}_i^n(d_{i,l}^n,b_{i,l}^n),g_n)=m_{i,l}/2
$$
and a sequence $\{\widetilde d_{i,l}^n\}$ where $1\leqslant i \leqslant \tilde s$ and $1 \leqslant l \leqslant \tilde t_i$ such that 
$$
\lim_{n \to \infty} \Vol(\widetilde{\mathcal{C}}_i^n(a_{i,l}^n,\widetilde d_{i,l}^n),g_n)= \lim_{n \to \infty} \Vol(\widetilde{\mathcal{C}}_i^n(\widetilde d_{i,l}^n,b_{i,l}^n),g_n)= \widetilde m_{i,l}/2.
$$
For $1\leqslant i \leqslant \tilde s$ let $\widetilde q^n_{i,l}\ll a^n_{i,l}$, $\widetilde q^n_{i,l}\to+\infty$. Consider the conformal maps\\ $\widetilde\Psi_{i,l}^n\colon \left(\widetilde{\mathcal{C}}_i^n(a_{i,l}^n-\widetilde q^n_{i,l},b_{i,l}^n+\widetilde q^n_{i,l}),c_n\right) \to (\mathbb{S}^2,[g_{can}])$ defined as
$$
\widetilde \Psi_{i,l}^n(t,\theta)=\frac{1}{e^{2(t-\widetilde d_{i,l}^n)}+1}(2e^{t-\widetilde d_{i,l}^n}\cos\theta, 2e^{t-\widetilde d_{i,l}^n}\sin\theta,e^{2(t-\widetilde d_{i,l}^n)}-1).
$$
Let $\widetilde \Omega_{i,l}^n\subset\mathbb{S}^2$ be the image of this map.
Let $\widetilde \Phi_{i,l}^n=\Phi_n \circ (\widetilde \Psi_{i,l}^n)^{-1}:(\widetilde \Omega_{i,l}^n, g_{can}) \to (\mathbb{S}^N, g_{can})$. Then $\widetilde \Phi_{i,l}^n$ is harmonic since $\Phi_n$ is harmonic and $\widetilde \Psi_{i,l}^n$ is conformal.
Moreover, it is shown in \cite{petrides2017existence} that the measure $\boldsymbol{1}_{\widetilde\Omega_{i,l}^n}|\nabla \widetilde \Phi_{i,l}^n|^2_{g_{can}}dv_{g_{can}}$ does not concentrate at the poles $(0,0,1)$ and $(0,0,-1)$ of $\mathbb{S}^2$. Indeed, if the measure concentrated at the poles then one would obtain a contradiction with the maximality of $\sum^s_{i=1}t_i+\sum^{\tilde s}_{i=1}{\tilde t_i}$.

Similarly, for $1\leqslant i \leqslant s$ if $a^n_{i,l}\ne 0$ let $0<q^n_{i,l}\ll a^n_{i,l}$, $q^n_{i,l}\to+\infty$, otherwise let $0<q^n_{i,l}\ll b^n_{i,l}$, $q^n_{i,l}\to+\infty$. If $a^n_{i,l}\ne 0$ consider the conformal maps $\Psi_{i,l}^n\colon \left({\mathcal{C}}_i^n(a_{i,l}^n-q^n_{i,l},b_{i,l}^n+ q^n_{i,l}),c_n\right) \to (\mathbb{S}^2,[g_{can}])$ defined on the orientable double covers as
$$
 \Psi_{i,l}^n(t,\theta)=\frac{1}{e^{2(t- d_{i,l}^n)}+1}(2e^{t- d_{i,l}^n}\cos\theta, 2e^{t- d_{i,l}^n}\sin\theta,e^{2(t- d_{i,l}^n)}-1).
$$
If $a^n_{i,l}= 0$, then $\Psi_{i,l}^n$ is defined in the same way, the only difference is that the domain is ${\mathcal{C}}_i^n(a_{i,l}^n,b_{i,l}^n+ q^n_{i,l})$. Either way, let $ \Omega_{i,l}^n\subset\mathbb{RP}^2$ be the image of this map.
Let $\Phi_{i,l}^n=\Phi_n \circ ( \Psi_{i,l}^n)^{-1}:(\Omega_{i,l}^n, g_{can}) \to (\mathbb{S}^N, g_{can})$. Then $\widetilde \Phi_{i,l}^n$ is harmonic since $\Phi_n$ is harmonic and $\Psi_{i,l}^n$ is conformal.
Similarly to the previous paragraph one has that the measure $\boldsymbol{1}_{\Omega_{i,l}^n}|\nabla \Phi_{i,l}^n|^2_{g_{can}}dv_{g_{can}}$ does not concentrate at the antipodal image of the pole $(0,0,1)$ in $\mathbb{RP}^2$. 


The exactly same procedure can be carried out for components $M_j^n(\alpha)$, $j\in J$. The only difference is that now we use the restriction of diffeomorphisms $\Psi^n$ given by Proposition~\ref{D-M} instead of the explicit harmonic map as above. As a result, one obtains domains $\check\Omega^n_j\subset M_\infty$ and harmonic maps $\check\Phi^n_j\colon\check\Omega^n_j\to\mathbb{S}^N$ such that the measure $\boldsymbol{1}_{\check\Omega_j^n}|\nabla\check \Phi_{j}^n|^2_{g_{can}}dv_{g_{can}}$ does not concentrate at the marked points of $\widehat{M_\infty}$.


As the next step, one applies bubble convergence theorem for harmonic maps and the non-concentration results above to choose a subsequence such that the measures $\boldsymbol{1}_{\widetilde\Omega_{i,l}^n}|\nabla \widetilde \Phi_{i,l}^n|^2_{g_{can}}dv_{g_{can}}$, $\boldsymbol{1}_{\Omega_{i,l}^n}|\nabla \Phi_{i,l}^n|^2_{g_{can}}dv_{g_{can}}$ and $\boldsymbol{1}_{\check\Omega_j^n}|\nabla\check \Phi_{j}^n|^2_{g_{can}}dv_{g_{can}}$ converge in $*$-weak topology.
One then uses eigenfunctions of limiting measures (and eigenfunctions on bubbles of $\{\Psi_j^n\}$ if bubbles exist) as test-functions for $\lambda_k(M,g_n)$. Since bubble convergence does not require the domain to be orientable and the construction of eigenfunctions supported on bubbles is local, this argument carries over to the non-orientable case without any changes. For further details, see~\cite[Section 7]{petrides2017existence}.

As a result, one obtains the following inequality
\begin{equation*}
\begin{split}
&\limsup_{n\to\infty}\Lambda_k(M,c_n)\leqslant \\ & \sum_{j\in J}\Lambda_{k_j}(\widehat{M^\infty_j},c_\infty) + \sum_{i=1}^{\widetilde s}\sum_{l=1}^{\widetilde t_i}\Lambda_{\widetilde r_{i,l}}(\mathbb{S}^2) + \sum_{i=1}^s\left(\sum_{l=1}^{t_i-1}\Lambda_{r_{i,l}}(\mathbb{S}^2) + \Lambda_{r_{i,t_i}}(S_i)\right),
\end{split}
\end{equation*}
where $S_i=\mathbb{RP}^2$ if the sequence $\{a^n_{i,t_i}\}_n$ contains infinitely many zeros, $S_i=\mathbb{S}^2$ otherwise, and
$$
\sum_{j\in J} k_j + \sum_{i=1}^{\widetilde s}\sum_{l=1}^{\widetilde t_i}\widetilde r_{i,l} + \sum_{i=1}^s\sum_{l=1}^{t_i}r_{i,l}\leqslant k.
$$
Finally, an application of inequality~\eqref{colbois} allows us to group together the terms with the same index $i$ to obtain inequality~\eqref{aim'}.

{\bf Case 2.} Assume that up to a choice of  a subsequence the following inequality holds
\begin{align*}
\Lambda_k(M, c_n) \leqslant \Lambda_{k-1}(M, c_n) +8\pi
\end{align*}
then we prove inequality (\ref{aim'}) by induction.

Note that if $k=1$ then by Theorem~\ref{roman} $\Lambda_1(M, [h_n]) > 8\pi$, i.e. $k=1$ falls under Case 1. Therefore, the inequality~\eqref{aim'} holds for $k=1$. This is the base of induction.

Suppose that the proposition holds for all numbers $k'\leqslant k$. We show that it also holds for $k+1$. Indeed, one has
\begin{align*}
\Lambda_{k+1}(M, c_n) \leqslant \Lambda_{k}(M, c_n)+8\pi=\Lambda_{k}(M, c_n)+\Lambda_1(\mathbb{S}^2)
\end{align*}
and inequality (\ref{aim'}) holds then we get 

\begin{gather*}
\begin{split}
&\lim_{n \to \infty} \Lambda_{k+1} (M, c_n) \leqslant\\ &\max \Big(\sum^{\widetilde m}_{i=1} \Lambda_{\widetilde k_i}(\widetilde\Sigma_{\widetilde\gamma_{i}}, c_\infty)+\sum^{m}_{i=1} \Lambda_{k_{i}}(\Sigma_{\gamma_{i}},c_\infty) + \sum_{i=1}^{\widetilde s}\Lambda_{\widetilde r_i}(\mathbb{S}^2) + \sum_{i=1}^s\Lambda_{r_i}(\mathbb{RP}^2) \Big)+\Lambda_1(\mathbb{S}^2) \leqslant \\& \leqslant \max \Big(\sum^{\widetilde m}_{i=1} \Lambda_{\widetilde k_i'}(\widetilde\Sigma_{\widetilde\gamma_{i}}, c_\infty)+\sum^{m}_{i=1} \Lambda_{k_{i}'}(\Sigma_{\gamma_{i}},c_\infty) + \sum_{i=1}^{\widetilde s}\Lambda_{\widetilde r_i'}(\mathbb{S}^2) + \sum_{i=1}^s\Lambda_{r_i'}(\mathbb{RP}^2) \Big),
\end{split}
\end{gather*}
where the term $\Lambda_1(\mathbb{S}^2)$ was absorbed by one of the terms inside $\max$ using inequality~\eqref{colbois}, and the last maximum is taken over all possible combinations of indices such that
 $$
 \sum_{i=1}^m k_i' + \sum_{i=1}^{\widetilde m}\widetilde k_i' + \sum_{i=1}^s r_i' + \sum_{i=1}^{\widetilde s}\widetilde r_i' = k+1.
 $$

\subsection{Non-hyperbolic case.} 
If $M=\mathbb{KL}$ or $M=\mathbb{T}^2$ the proof is very similar. Indeed, as it follows from the discussion in Section~\ref{non-neg} for degenerating sequence one can find a collapsing geodesic and the whole surface becomes a flat collar of width $w_n\to+\infty$. An analog of Proposition~\ref{sequences} is proved in exactly the same way. The only difference in the rest of the proof is the fact that there is at most one domain $M^n_j(\alpha^n)$ and it is a flat cylinder or a M\"obius band. Therefore, to construct $\check\Phi^n_j$ instead of the Deligne-Mumford compactification one uses the same construction as for $\widetilde\Phi^n_{i,l}$ or $\Phi^n_{i,l}$.

\bibliography{mybib}
\bibliographystyle{alpha}

\end{document}

%% file: A.tex
\begingroup%
  \makeatletter%
  \providecommand\color[2][]{%
    \errmessage{(Inkscape) Color is used for the text in Inkscape, but the package 'color.sty' is not loaded}%
    \renewcommand\color[2][]{}%
  }%
  \providecommand\transparent[1]{%
    \errmessage{(Inkscape) Transparency is used (non-zero) for the text in Inkscape, but the package 'transparent.sty' is not loaded}%
    \renewcommand\transparent[1]{}%
  }%
  \providecommand\rotatebox[2]{#2}%
  \ifx\svgwidth\undefined%
    \setlength{\unitlength}{829.30522769bp}%
    \ifx\svgscale\undefined%
      \relax%
    \else%
      \setlength{\unitlength}{\unitlength * \real{\svgscale}}%
    \fi%
  \else%
    \setlength{\unitlength}{\svgwidth}%
  \fi%
  \global\let\svgwidth\undefined%
  \global\let\svgscale\undefined%
  \makeatother%
  \begin{picture}(1,0.19576713)%
    \put(0.27847701,0.10844456){\color[rgb]{0,0,0}\makebox(0,0)[lt]{\begin{minipage}{0.04586808\unitlength}\raggedright \end{minipage}}}%
    \put(0,0){\includegraphics[width=\unitlength,page=1]{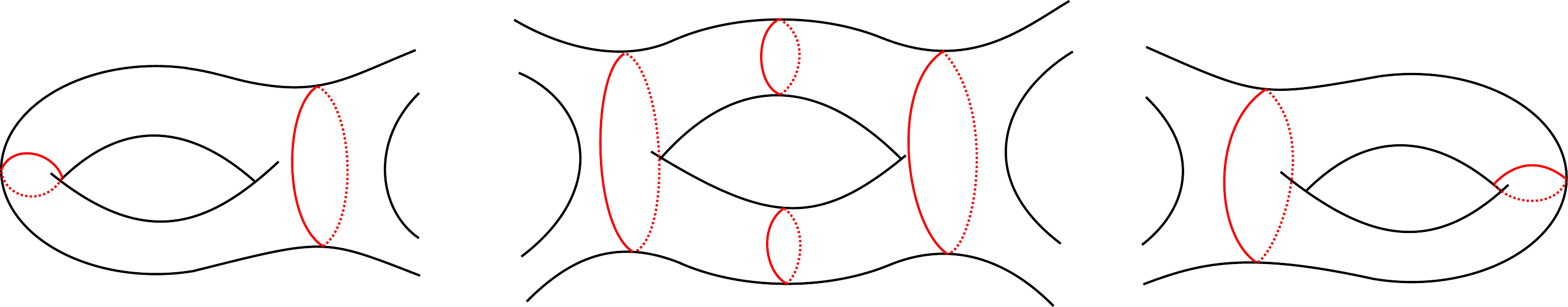}}%
    \put(0.27852756,0.09719371){\color[rgb]{0,0,0}\makebox(0,0)[lt]{\begin{minipage}{0.20969025\unitlength}\raggedright {\fontsize{18}{18}$\ldots$}\end{minipage}}}%
    \put(0,0){\includegraphics[width=\unitlength,page=2]{1.pdf}}%
    \put(0.67645101,0.09719371){\color[rgb]{0,0,0}\makebox(0,0)[lt]{\begin{minipage}{0.20969025\unitlength}\raggedright {\fontsize{18}{18}$\ldots$}\end{minipage}}}%
    \put(0,0){\includegraphics[width=\unitlength,page=3]{1.pdf}}%
  \end{picture}%
\endgroup%

%% file: B.tex
\begingroup%
  \makeatletter%
  \providecommand\color[2][]{%
    \errmessage{(Inkscape) Color is used for the text in Inkscape, but the package 'color.sty' is not loaded}%
    \renewcommand\color[2][]{}%
  }%
  \providecommand\transparent[1]{%
    \errmessage{(Inkscape) Transparency is used (non-zero) for the text in Inkscape, but the package 'transparent.sty' is not loaded}%
    \renewcommand\transparent[1]{}%
  }%
  \providecommand\rotatebox[2]{#2}%
  \ifx\svgwidth\undefined%
    \setlength{\unitlength}{838.19308063bp}%
    \ifx\svgscale\undefined%
      \relax%
    \else%
      \setlength{\unitlength}{\unitlength * \real{\svgscale}}%
    \fi%
  \else%
    \setlength{\unitlength}{\svgwidth}%
  \fi%
  \global\let\svgwidth\undefined%
  \global\let\svgscale\undefined%
  \makeatother%
  \begin{picture}(1,0.16994227)%
    \put(0.18643631,0.10840526){\color[rgb]{0,0,0}\makebox(0,0)[lt]{\begin{minipage}{0.14920561\unitlength}\raggedright {\fontsize{14}{14}$\ldots$}\end{minipage}}}%
    \put(0,0){\includegraphics[width=\unitlength,page=1]{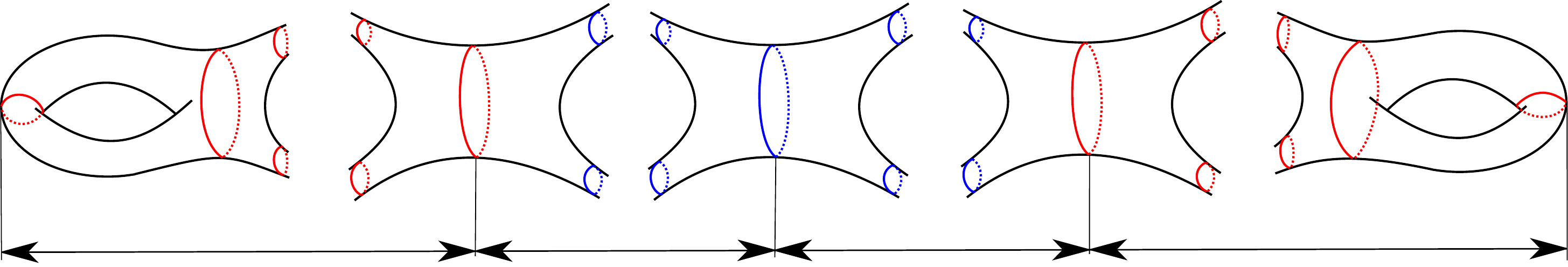}}%
    \put(0.11428424,0.05448719){\color[rgb]{0,0,0}\makebox(0,0)[lt]{\begin{minipage}{0.21150842\unitlength}\raggedright  {\fontsize{14}{14}$\frac{\gamma-\gamma'}{2}$}\end{minipage}}}%
    \put(0.22404636,0.18767206){\color[rgb]{0,0,0}\makebox(0,0)[lt]{\begin{minipage}{0.51245684\unitlength}\raggedright  \end{minipage}}}%
    \put(0.37881441,0.10840526){\color[rgb]{0,0,0}\makebox(0,0)[lt]{\begin{minipage}{0.14920561\unitlength}\raggedright {\fontsize{14}{14}$\ldots$}\end{minipage}}}%
    \put(0.57566643,0.10840526){\color[rgb]{0,0,0}\makebox(0,0)[lt]{\begin{minipage}{0.14920561\unitlength}\raggedright {\fontsize{14}{14}$\ldots$}\end{minipage}}}%
    \put(0.77251845,0.10840526){\color[rgb]{0,0,0}\makebox(0,0)[lt]{\begin{minipage}{0.14920561\unitlength}\raggedright {\fontsize{14}{14}$\ldots$}\end{minipage}}}%
    \put(0.38719283,0.05625061){\color[rgb]{0,0,0}\makebox(0,0)[lt]{\begin{minipage}{0.11992629\unitlength}\raggedright  {\fontsize{14}{14}$\frac{\gamma'}{2}$}\end{minipage}}}%
    \put(0.82116205,0.05448719){\color[rgb]{0,0,0}\makebox(0,0)[lt]{\begin{minipage}{0.11992629\unitlength}\raggedright  {\fontsize{14}{14}$\frac{\gamma-\gamma'}{2}$}\end{minipage}}}%
    \put(0.57957093,0.05448719){\color[rgb]{0,0,0}\makebox(0,0)[lt]{\begin{minipage}{0.11992629\unitlength}\raggedright  {\fontsize{14}{14}$\frac{\gamma'}{2}$}\end{minipage}}}%
    \put(0,0){\includegraphics[width=\unitlength,page=2]{2.pdf}}%
  \end{picture}%
\endgroup%